\documentclass{amsart}

\usepackage{mathtools,amsmath, amsfonts, amssymb, latexsym, amsthm,
  amscd, cancel, enumerate, rotating, comment, appendix,
  makecell, paralist,times,graphicx,pgf,geometry,dirtytalk,
appendix}
\usepackage[colorlinks=true, pdfstartview=FitV, linkcolor=blue, citecolor=blue, urlcolor=blue]{hyperref}
\usepackage[capitalise, noabbrev]{cleveref}
\usepackage{paralist}
\usepackage{pgf}
\usepackage{times}
\usepackage{enumitem}
\usepackage{caption}
\usepackage{subcaption}
\usepackage[T1]{fontenc}

\newcommand{\R}{{\sf{Div}}}

\usepackage{xcolor}
\usepackage{tikz}
\usepackage{tikz-3dplot}
\usetikzlibrary{shapes}
\usetikzlibrary{snakes}
\usetikzlibrary{arrows}
\usepackage{tikz-cd}
\usetikzlibrary{matrix, decorations.pathmorphing}

\usetikzlibrary{calc}
\usetikzlibrary{decorations.markings}

\tikzset{
  midarrow/.style={
    decoration={markings,mark=at position 0.5 with {\arrow{>}}},
    postaction={decorate}
  }
}

\newif\ifdviwin

\numberwithin{equation}{section}
\definecolor{gold}{rgb}{0.9, 0.7, 0.2}

\theoremstyle{plain}
\newtheorem{theorem}{Theorem}[section]

\newtheorem*{TheoremA}{Theorem A}
\newtheorem*{TheoremB}{Theorem B}

\newtheorem*{Main Theorem}{Main Theorem}

\newtheorem{proposition}[theorem]{Proposition}
\newtheorem{lemma}[theorem]{Lemma}
\newtheorem{corollary}[theorem]{Corollary}

{\Alph{theorem}}
{\Alph{theorem}}

\theoremstyle{definition}
\newtheorem{notation}[theorem]{Notation}
\newtheorem{setup}[theorem]{Setup}

\newtheorem{remark}[theorem]{Remark}

\newtheorem{definition}[theorem]{Definition}
\newtheorem{example}[theorem]{Example}

\newtheorem{question}[theorem]{Question}

\newtheorem{Remark}[theorem]{Remark}

  \newcounter{numlist} %
  {\end{list}}%

\theoremstyle{remark}

\newtheorem{claim}{Claim}

\newtheorem{chunk}[theorem]{}

\numberwithin{equation}{theorem}

\newcommand{\st}{\ \colon \ }

\newcommand{\pd}{\mathrm{pd}}
\newcommand{\Taylor}{\mathrm{Taylor}}

\newcommand{\lcm}{\mathrm{lcm}}

\newcommand{\e}{\epsilon}
\newcommand{\ed}[1]{\epsilon_{\mbox{\tiny$\D$},#1}}
\newcommand{\LL}{\mathbb{L}}
\newcommand{\TT}{\mathbb{T}}
\newcommand{\tuple}[1]{\langle #1 \rangle}

\newcommand{\ssm}{\smallsetminus}

\newcommand{\E}{\mathcal{E}}
\newcommand{\NN}{\mathbb{N}}
\newcommand{\RR}{\mathbb{R}}
\newcommand{\ba}{\mathbf{a}}

\newcommand{\be}{\mathbf{e}}

\newcommand{\qwhere}{\quad \mbox{ where } \quad }
\newcommand{\qand}{\quad \mbox{ and } \quad }
\newcommand{\qfor}{\quad \mbox{ for } \quad }

\newcommand{\qif}{\quad \mbox{if} \quad}
\newcommand{\qwith}{\quad \mbox{with} \quad}
\newcommand{\qwhen}{\quad \mbox{when} \quad}
\newcommand{\qforsome}{\quad \mbox{for some} \quad}
\newcommand{\qforall}{\quad \mbox{for all} \quad}

\newcommand{\ED}{\E_{q,\D}}
\newcommand{\EDr}{\ED^{\,\,r}}
\newcommand{\EDsq}{\ED^{\,\,2}}
\newcommand{\Eqsq}{\E_q^{\,\,2}}

\newcommand{\LDsq}{\LD^2}

\newcommand{\LqD}{\LL_{q,\D}}

\newcommand{\LqDsq}{\LqD^2}

\renewcommand{\LDsq}{\LqDsq}

\newcommand{\Nrq}{\N^r_q}
\newcommand{\pme}{{\pmb{\e}}}
\newcommand{\ped}{\pme_{\mbox{\tiny$\D$}}}

\newcommand{\N}{\mathcal{N}}
\newcommand{\U}{\mathcal{U}}

\newcommand{\M}{\mathcal{M}}
\newcommand{\D}{\mathcal{D}}

\newcommand{\sfk}{\mathsf k}

\begin{document}
\author[S.~M.~Cooper]{Susan M. Cooper}
\address{Department of Mathematics\\
University of Manitoba\\
Winnipeg, MB\\
Canada R3T 2N2}
\email{susan.cooper@umanitoba.ca}

\author[S.~El Khoury]{Sabine El Khoury}
\address{Applied Mathematics and Computational Science program,
King Abdullah University Of Science and Technology,
Thuwal,  KSA 23955}
\email{sabine.khoury@kaust.edu.sa}

\author[S.~Faridi]{Sara Faridi}
\address{Department of Mathematics \& Statistics\\
Dalhousie University\\
6316 Coburg Rd.\\
PO BOX 15000\\
Halifax, NS\\
Canada B3H 4R2}
\email{faridi@dal.ca}

\author[S.~Morey]{Susan Morey}
\address{Department of Mathematics\\
Texas State University\\
601 University Dr.\\
San Marcos, TX 78666\\U.S.A.}
\email{morey@txstate.edu}

\author[L.~M.~\c{S}ega]{Liana M.~\c{S}ega}
\address{Division of Computing, Analytics and Mathematics, 
University of Missouri-Kansas City, Kansas City, MO 64110 U.S.A.}
\email{segal@umkc.edu}

\author[S.~Spiroff]{Sandra Spiroff }
\address{Department of Mathematics,
University of Mississippi,
Hume Hall 335, P.O. Box 1848, University, MS 38677
USA}
\email{spiroff@olemiss.edu}

\setlist{font=\normalfont}

\keywords{powers of ideals; simplicial complex; Betti numbers; free
  resolutions; monomial ideals; extremal ideals}

\subjclass[2010]{13D02; 13F55; 05E40}

\title{ Cellular resolutions of second powers of square-free monomial ideals with divisibility relations}

\begin{abstract}
Using divisibility relations between the generators of a square-free monomial ideal $I$, we describe divisibility relations between the generators of the second power $I^2$. We then employ  discrete Morse theory to produce a cellular free resolution of $I^2$ which is minimal for specific ideals that are  extremal with respect to a given divisibility relation. In particular, we provide sharp bounds on the projective dimension of $I^2$ when the generators of $I$ satisfy at least one divisibility relation.  \end{abstract} 
\maketitle

\section{\bf Introduction}
A free resolution of an ideal generated by a set of polynomials, introduced in the nineteenth century by David Hilbert, is a structure that encodes all the relations between those polynomials via a sequence of free modules or vector spaces. 
An effective tool for constructing free resolutions of monomial ideals is homogenizing boundary maps of topological structures. This point of view allows one to translate questions about bounds on Betti numbers of monomial ideals into bounds on enumerating cells of acyclic cell complexes. 

The earliest results in this direction are due to Diana Taylor~\cite{T}: every monomial ideal $I$ of a polynomial ring generated by $q$ monomials  has a free resolution -- known today as  {\it Taylor's resolution} -- supported on a simplex on $q$ vertices.  Since the minimal free resolution of a monomial ideal is unique, Taylor's resolution provides an upper bound for the Betti numbers of $I$ as  $\beta_i(I) \leq \binom{q}{i+1}$. The binomial bound provided by Taylor's resolution is sharp in the sense that  there are ideals whose Betti numbers are exactly those binomial coefficients. For example, an ideal generated by $q$ distinct variables $I=(x_1,\ldots,x_q)$ will satisfy $\beta_i(I) = \binom{q}{i+1}$.

However in general, the binomial coefficients tend to be too large to be effective bounds for Betti numbers. This is because the Betti numbers keep track of the minimal number of relations between the generators of the ideal, and typically any $q$ monomials  are more inter-related than $q$ distinct variables are. So for an ideal $I$ generated by, say, three monomials $m_1,m_2,m_3$,  the simplex labeled and shown in \cref{f:two-edges} on the left always supports a resolution. If, additionally, we know that $m_1 \mid \lcm(m_2,m_3)$, then the graph on the right in \cref{f:two-edges} also supports a (in fact minimal) free resolution of $I$ (see~\cite{BS,F14}). Moreover, the bounds on the Betti numbers provided by the graph are smaller than the binomial bounds offered by the Taylor simplex on the left. 
As the size of the ideal grows, the reductions in these bounds can be quite significant.

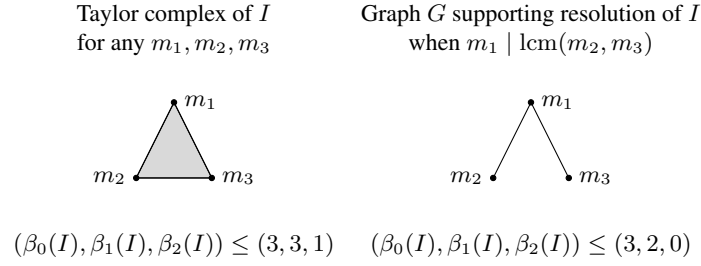
\begin{figure}[!htbp]
$$
{\small \begin{array}{cc}
\mbox{Taylor complex of } I & 
\mbox{Graph $G$ supporting resolution of } I \\
\mbox{for any } m_1, m_2, m_3
&\mbox{when } m_1 \mid \lcm(m_2,m_3)\\
&\\
\begin{tikzpicture}
\tikzstyle{point}=[inner sep=0pt]

\node (b)[point,label=left:$m_2$] at (-.5,0){};
\node (c)[point,label=right:$m_1$] at (0,1){};
\node (d)[point,label=right:$m_3$] at (.5,0){};

\draw  [fill=gray!30] (d.center) -- (b.center) -- (c.center) -- cycle;
\draw (b.center) -- (c.center);
\draw (b.center) -- (d.center);
\draw (c.center) -- (d.center);

\filldraw [black] (b.center) circle (1pt);
\filldraw [black] (c.center) circle (1pt);
\filldraw [black] (d.center) circle (1pt);
\end{tikzpicture}
&
\begin{tikzpicture}
\tikzstyle{point}=[inner sep=0pt]
\node (b)[point,label=left:$m_2$] at (-.5,0){};
\node (c)[point,label=right:$m_1$] at (0,1){};
\node (d)[point,label=right:$m_3$] at (.5,0){};

\draw (b.center) -- (c.center);
\draw (c.center) -- (d.center);

\filldraw [black] (b.center) circle (1pt);
\filldraw [black] (c.center) circle (1pt);
\filldraw [black] (d.center) circle (1pt);
\end{tikzpicture}\\

&\\
(\beta_0(I),\beta_1(I),\beta_2(I))\leq (3,3,1)&
(\beta_0(I),\beta_1(I),\beta_2(I))\leq (3,2,0)
\end{array}
}
$$
\caption{Simplicial resolutions of $I=(m_1,m_2,m_3)$}\label{f:two-edges}
\end{figure}

If we take the square $I^2$ of a monomial ideal with $q \geq 2$ generators, Cooper et al~\cite{L2} observed that the Taylor resolution will never be minimal, making the binomial bounds given by the Taylor simplex unachievably large. As an alternative, when $I$ is square-free, the authors of~\cite{L2} offered the simplicial complex $\LL^2_q$ (\cref{f:L2-box}), which always supports the free resolution of  $I^2$, and when  $I$ is the {\it extremal} ideal $\E_q$~\cite{Lr}, $\LL^2_q$ supports a {\it minimal} resolution of $\Eqsq$. The complex $\LL^2_q$ is significantly smaller than the Taylor simplex on the same number of vertices, giving tighter bounds on Betti numbers, and the existence of the extremal ideal makes these bounds sharp. 

The  same authors observed in~\cite{Morse,koszul} that if $I$ has a (minimal) free resolution supported on a graph $G$ on $q$ vertices, as in the case for the ideal on the right in \cref{f:two-edges}, then $I^2$ would have a (minimal) free resolution supported on a cell complex that can be combinatorially constructed as $G^2$. For the example in \cref{f:two-edges} the cell complex $G^2$ is shown on the right in \cref{f:L2-box}. 

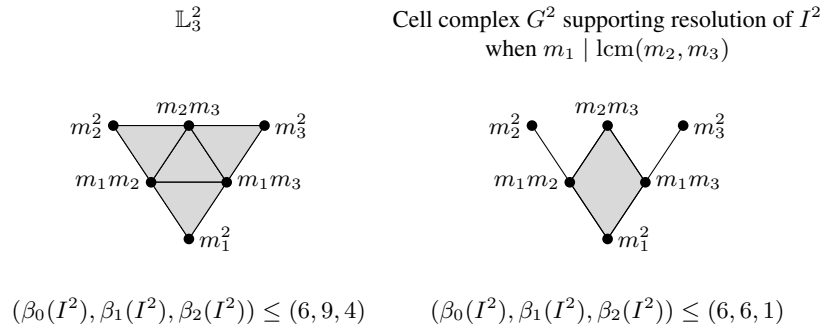
\begin{figure}[!htbp]
$$
{\small
\begin{array}{cc}
\LL^2_3& 
\mbox{Cell complex $G^2$ supporting resolution of } I^2 \\
&\mbox{when } m_1 \mid \lcm(m_2,m_3)\\
&\\
\begin{tikzpicture}
\tikzstyle{point}=[inner sep=0pt]
\node (A)[point,label=left:$m_1m_2$] at (-.5,0){};
\node (B)[point,label=above:$m_2m_3$] at (0,.75){};
\node (C)[point,label=right:$m_1m_3$] at (.5,0){};
\node (D)[point,label=left:$m_2^2$] at (-1,.75){};
\node (E)[point,label=right:$m_1^2$] at (0,-.75){};
\node (F)[point,label=right:$m_3^2$] at (1,.75){};

   \draw (A) -- (B);
   \draw (B) -- (C);   
   \draw (C) -- (A);
  \draw  [fill=gray!30] (A.center) -- (B.center) -- (C.center) -- cycle;

  \draw (A) -- (B);
   \draw (B) -- (D);   
   \draw (D) -- (A);
  \draw  [fill=gray!30] (A.center) -- (B.center) -- (D.center) -- cycle;

  \draw (A) -- (E);
   \draw (E) -- (C);   
   \draw (C) -- (A);
  \draw  [fill=gray!30] (A.center) -- (E.center) -- (C.center) -- cycle;

  \draw (F) -- (B);
   \draw (B) -- (C);   
   \draw (C) -- (F);
  \draw  [fill=gray!30] (F.center) -- (B.center) -- (C.center) -- cycle;

\foreach \point in {A, B, C, D, E, F} \fill (\point) circle (2pt);
 \end{tikzpicture}
 &
\begin{tikzpicture}
\tikzstyle{point}=[inner sep=0pt]
\node (A)[point,label=left:$m_1m_2$] at (-.5,0){};
\node (B)[point,label=above:$m_2m_3$] at (0,.75){};
\node (C)[point,label=right:$m_1m_3$] at (.5,0){};
\node (D)[point,label=left:$m_2^2$] at (-1,.75){};
\node (E)[point,label=right:$m_1^2$] at (0,-.75){};
\node (F)[point,label=right:$m_3^2$] at (1,.75){};

  \draw  [fill=gray!30] (A.center) -- (B.center) -- (C.center) -- (E.center)--cycle;
  
   \draw (A) -- (D);
   \draw (A)-- (B);
   \draw (B) -- (C);
   \draw(C) -- (E);
   \draw (E) -- (A);
  \draw (F) -- (C);
\foreach \point in {A, B, C, D, E, F} \fill (\point) circle (2pt);
 \end{tikzpicture}\\
 &\\
 (\beta_0(I^2),\beta_1(I^2),\beta_2(I^2))\leq (6,9,4)&
(\beta_0(I^2),\beta_1(I^2),\beta_2(I^2))\leq (6,6,1)
\end{array}
}
$$
\caption{Cellular resolutions of $I^2=(m_1,m_2,m_3)^2$}\label{f:L2-box}
\end{figure}

With this mindset, a natural question is the following one.

\begin{question}\label{q:power-complex}
If a simplicial complex $\Delta$ supports a (minimal) free resolution of a monomial ideal $I$, then can we give a combinatorial definition of  a cell complex 
$\Delta^r$ which supports a (minimal) free resolution of $I^r$?
\end{question}

The simplicial complex $\Delta$ mentioned in \cref{q:power-complex} can be a strict subcomplex of the Taylor simplex when there is a \say{divisibility relation} such as $m_1 \mid \lcm(m_2,\ldots,m_s)$ between (some of) the monomial generators of $I$. Such a  relation can be formalized by the indices of the monomials as $(1,\{2,\ldots,s\})$ (\cite{Divisibilities}).  For example, the divisibility relation $(1, \{2,3\})$ between the generators of $I$ in the right column of \cref{f:two-edges} allowed us to improve the bounds on the Betti numbers of both $I$ and $I^2$ (\cref{f:L2-box}).

When there is a set $\D$ of divisibility relations between the $q$ square-free monomials that generate $I$, the Taylor simplex and the simplicial complex $\LL_q^2$ are both too large in the sense that there is {\it no ideal} satisfying the divisibility relations of $\D$ for which the resulting resolutions are minimal. This leads to a refinement of \cref{q:power-complex}, which will be the theme of investigation for this paper.

\begin{question}\label{q:D} Let $I$ be an ideal generated by $q$ square-free monomials and let $\D$ be a set of divisibility relations between those monomials. Suppose $\Delta$ is  a simplicial complex that supports a free resolution of $I$.
\begin{enumerate}
\item  How can one prune $\LL^2_q$ using $\D$ to find a cell complex $\Delta^2$ that supports a free resolution of $I^2$?
\item If $\Delta$ supports a {\it minimal} resolution of $I$, when does $\Delta^2$ in ~(1) support a {\it minimal} resolution of $I^2$? 
\end{enumerate}
\end{question}

Our goal in this paper is to answer \cref{q:D} when $\D$ contains at least one nontrivial divisibility relation, for example when $m_1 \mid \lcm (m_2,\ldots,m_s)$ for some $s\ge 3$. The way we achieve our goal is by  making use of the $\D$-extremal ideal $\ED$, defined by the authors  in~\cite{Divisibilities}. These ideals are called extremal because their powers have the maximal Betti numbers among all powers of ideals generated by $q$ square-free monomials that satisfy the relations in $\D$. In particular, it was shown in \cite[Theorem 4.4]{Divisibilities} that if $I$ is any square-free monomial ideal satisfying the divisibility relations in $\D$ 
and $\Delta$ is any simplicial complex supporting a free resolution of $\EDr$, then $\Delta$ also supports a free resolution of $I^r$. 

If $q=3$ and $\D=\{(1,\{2,3\})\}$, 
then we can take $\Delta$ and $\Delta^2$ to be the complexes on the right side of \cref{f:two-edges,f:L2-box}, respectively. If $q=4$ with the same $\D=\{(1,\{2,3\})\}$,
then the $\Delta$ on the left in \cref{f:two-triangles-edge} supports a free resolution of $I$, and we prove in this paper that  $\Delta^2$ is the cell complex whose 1-skeleton is depicted on the right and whose higher dimensional cells are described in \cref{s:geometric} (see \cref{f:L24-cut-down}).   It is worth highlighting that deleting all vertices involving $m_4$ in the complexes in \cref{f:two-triangles-edge} produces the graph $G$ and the cell complex $G^2$ in  \cref{f:two-edges,f:L2-box}.

\begin{figure}[!htbp]
{\small
\begin{tabular}{ccc}
$\Delta$ supports resolution of $I$ &
& 
$\Delta^2$  supports resolution of $I^2$  \\
\raisebox{1cm}{
\begin{tikzpicture}
\tikzstyle{point}=[inner sep=0pt]
\node (a)[point,label=left: $m_2$] at (-1,1) {};
\node (b)[point,label=right:$m_1$] at (0,0){};
\node (c)[point,label=right:$m_4$] at (0,2){};
\node (d)[point,label=right:$m_3$] at (1,1){};
\draw (a.center) -- (b.center);
\draw (a.center) -- (c.center);
\draw (a.center) -- (d.center);
\draw (b.center) -- (c.center);
\draw (b.center) -- (d.center);
\draw (c.center) -- (d.center);
\draw  [fill=gray!30] (a.center) -- (b.center) -- (c.center) -- cycle;
\draw  [fill=gray!30] (d.center) -- (b.center) -- (c.center) -- cycle;
\filldraw [black] (a.center) circle (1pt);
\filldraw [black] (b.center) circle (1pt);
\filldraw [black] (c.center) circle (1pt);
\filldraw [black] (d.center) circle (1pt);
\end{tikzpicture}}
&&  
{\tiny \tdplotsetmaincoords{60}{150}
\begin{tikzpicture}[tdplot_main_coords,scale=.7]
    \coordinate (A) at (0,0,0);      
    \coordinate (B) at (2,0,0);      
    \coordinate (C) at (4,0,0);      
    \coordinate (D) at (0,2,0);      
    \coordinate (E) at (0,4,0);      
    \coordinate (F) at (2,2,0);      
    \coordinate (G) at (0,0,2);      
    \coordinate (H) at (0,0,4);      
    \coordinate (I) at (2,0,2);      
    \coordinate (J) at (0,2,2);      
\draw (F) -- (B) ;
\draw (F) -- (D);
\draw (A) -- (B);
\draw (A) -- (D);   
\draw (A) -- (G);
\draw (G) -- (H);
\draw (G) -- (J);
\draw (G) -- (I);
\draw (B) -- (I);
\draw (G) -- (B);
\draw (B) -- (C);
\draw (I) -- (C);
\draw (H) -- (I);
\draw (H) -- (J); 
\draw (J) -- (D);  
\draw (E) -- (J);
\draw (F) -- (J);
\draw (G) -- (F);
\draw (I) -- (F);
\draw (G) -- (D);
\draw (D) -- (E);
    \foreach \point in {A, B, C, D, E, F, G, H, I, J} \fill (\point) circle (2pt);
 
    \coordinate (AA) at (1.1,.8,0.4);
    \coordinate (BB) at (2.6,.5,0);
    \coordinate (CC) at (3.5,0,0);
    \coordinate (DD) at (0.6,2.5,0);
    \coordinate (EE) at (0,3.5,0);
    \coordinate (FF) at (2.5,2.5,0);
    \coordinate (GG) at (0.7,0,2.4);
    \coordinate (HH) at (.5,0,4.25);
    \coordinate (II) at (1.75,0,1.75);
    \coordinate (JJ) at (0,1.75,1.75);    
    
    \node[right] at (AA) {\tiny{$m_1^2$}};
    \node at (BB) {\tiny{$m_1m_2$}};
    \node[above left] at (C) {\tiny{$m_2^2$}};
    \node  at (DD) {\tiny{$m_1m_3$}};
    \node[above right] at (E) {\tiny{$m_3^2$}};
    \node at (FF) {\tiny{$m_2m_3$}};
    \node at (GG) {\tiny{$m_1m_4$}};
    \node at (HH) {\tiny{$m_4^2$}};
    \node[left] at (I) {\tiny{$m_2m_4$}};
    \node[right] at (J) {\tiny{$m_3m_4$}};
    \draw[-] (0,0,0) -- (3.5,0,0) node[anchor=north east]{};
    \draw[-] (0,0,0) -- (0,3.5,0) node[anchor=north west]{};
    \draw[-] (0,0,0) -- (0,0,3.5) node[anchor=south]{};
\end{tikzpicture} }
  \end{tabular}
}
\caption{} \label{f:two-triangles-edge}
\end{figure}

To achieve our goal,  we offer a specific Morse matching on the faces of the $\LL^2_q$ based on one divisibility relation $(1,\{2,\ldots,s\})$, and we  use this matching to prune the extra faces of $\LL^2_q$. The end result is a cell complex $\LDsq$ that supports a free resolution of the square of any ideal generated by $q$ square-free monomials which satisfy $m_1 \mid \lcm(m_2,\ldots,m_s)$ for some $s \leq q$, and this resolution can be minimal.

\begin{TheoremA}[\cref{t:minimality}]  Suppose $\D=\{(1,\{2,\ldots,s\})\}$ is a set consisting of one divisibility relation, and let $I$ be an ideal generated by $q\geq s>0$ square-free monomials satisfying the divisibility relation in $\D$. 
Then there is a cellular complex $\LDsq$ on ${{q+1}\choose{2}}$ 
 vertices supporting a free resolution of $I^2$. Moreover, the resolution of $I^2$ supported on $\LDsq$ is minimal if $I$ is the $\D$-extremal ideal $\ED$.
\end{TheoremA}

As can be gleaned from our earlier discussions, the central player in our arguments will be the $\D$-extremal ideal $\ED$. Indeed, our Morse matching is built based on divisibility relations between the generators of $\EDsq$. 
Given a set of divisibility relations on $\leq q$ integers, the $\D$-extremal ideal $\ED$ was built in~\cite{Divisibilities} as an ideal with $q$ generators that satisfy only divisibility relations coming from $\D$, and no other divisibility relations~(\cite[Theorem 3.4]{Divisibilities}). We prove a similar statement in \cref{p:lcm-general} for $\EDsq$, and when $|\D| \leq 1$, characterizing all the divisibility relations that $\EDsq$ can satisfy. 

To reiterate, if $I$ is any ideal generated by $q$ square-free monomials where one generator divides the least common multiple of other generators (as is the case for most monomial ideals which are not complete intersections), then the generators of $I^2$ must satisfy, up to reindexing, the divisibility relations outlined in \cref{p:lcm-general}. Therefore the Morse matching that leads to a Morse complex $\LDsq$ is also based on divisibility relations satisfied by $I^2$, and for this reason, it supports a free resolution of $I^2$.

We then use the size of $\LDsq$ to find the projective dimensions of $\ED$ and $\EDsq$ when $|\D|=1$, and as a result we provide bounds on  the projective dimensions of the first and second powers of any square-free monomial ideals with a divisibility relation between their generators. 

\begin{TheoremB}[\cref{t:betti-2-bound}]
 Let $I$ be an ideal of a polynomial ring minimally generated by square-free monomials $m_1,\ldots,m_q$ such that $m_{i_1} \mid \lcm(m_{i_2},\ldots,m_{i_s})$ where $3 \leq s \leq q$ and 
$i_1,\ldots,i_s$  are distinct integers between $1$ and $q$. Then
$$
\pd(I) \leq q-2 \qand 
\pd(I^2) \leq \begin{cases}
\binom{q}{2}-(q-s+2)&\text{if $q>s$}\\
\binom{q}{2}-1  &\text{if $q=s$}.
\end{cases}
$$
Moreover, these bounds are sharp. 
\end{TheoremB}

The final section of the paper analyzes various examples and bounds on Betti numbers of squares of ideals using the cell complexes $\LDsq$.

\subsection*{Acknowledgements}
The research for this paper was initiated during the authors' stay at the American Institute of Mathematics (AIM), as part of the AIM SQuaRE program.  We are grateful to AIM for their warm hospitality, and for providing a stimulating  research environment.  Authors Cooper and Faridi are partially supported by NSERC Discovery Grants 2024-05444 and 2023-05929, respectively. Spiroff is supported by Simons Foundation \#584932.

\section{\bf Background} \label{background}

This section provides the background and notation that will be used
throughout the paper.
Let $R=\sfk[x_1, \ldots, x_n]$ be a polynomial ring over a
field $\sfk$ and let $I$ be an ideal of $R$. A {\bf free resolution} of \(I\) is an exact sequence of the form
\[
0 \rightarrow R^{\beta_t} \rightarrow R^{\beta_{t-1}} \rightarrow \cdots \rightarrow R^{\beta_1} \rightarrow R^{\beta_0} \rightarrow I \rightarrow 0,
\]
where \(R^{\beta_j}\) is a free \(R\)-module of rank \(\beta_j\) and \(t \in \mathbb{N}\). When \(\beta_j\) is the smallest possible rank of a free module in the \(j\)-th spot of any free resolution of \(I\) for each \(j\), the resolution is {\bf minimal}. In this case, the numbers \(\beta_j\) are invariants of \(I\) and are called the {\bf Betti numbers} of \(I\).

There has been a significant amount of interest in using combinatorics and topology to find concrete structures that describe free resolutions when $I$ is generated by monomials.  In her thesis~\cite{T}, Taylor described a free resolution of any ideal minimally generated by $q$ monomials using a special simplicial complex.  This complex can be used to obtain upper bounds for the Betti numbers.  We now recall the terminology needed to understand Taylor's ground-breaking work.

A {\bf simplicial complex}
  $\Delta$ over a vertex set $V$ is a set of subsets of $V$ that is closed under inclusion. That is, 
  if $F \in \Delta$ and $G \subseteq F$ then $G \in \Delta$. The {\bf induced subcomplex} of $\Delta$ on a subset $W \subseteq V$ is the subcomplex
  $$\Delta_W=\{ \sigma \in \Delta \mid \sigma \subseteq W\}.$$
  An element $\sigma$ of
  $\Delta$ is called a {\bf face} of $\Delta$, and the faces that are maximal with respect to 
  inclusion are called {\bf facets} of $\Delta$. A simplicial complex is
  uniquely determined by its facets.  We use the
  notation $$\Delta=\tuple{F_0,\ldots,F_q}$$ to describe a simplicial
  complex whose facets are $F_0,\ldots,F_q$.  A {\bf simplex} is a simplicial complex with a unique facet.   The {\bf dimension} of a face $F$ is defined to be $\dim(F) = |F|-1$ and the dimension of a simplicial complex $\Delta$ is the maximum of the dimensions of its faces.  The {\bf $f$-vector} of a simplicial complex $\Delta$ of dimension $d$ is the vector $(f_0, f_1, \ldots , f_{d+1})$ where $f_0=1$ and $f_i$ is the number of faces of cardinality $i$ for $i \ge 1$.

Let $I = (m_1, \ldots, m_q)$ be an ideal minimally generated by $q$ monomials.  The {\bf Taylor complex} of $(m_1,\ldots,m_q)$, which we denote by $\Taylor(I)$, is a simplex on $q$ vertices, with each vertex labeled by a monomial generator of $I$ and each face $\sigma$ labeled by a monomial $\lcm(\sigma)$,  which is the least common multiple (lcm) of the monomial labels of the vertices of $\sigma$. The {\bf lcm lattice} of $I$, LCM($I$), is the set of all least common multiples of subsets of the minimal monomial generating set of $I$ partially ordered by division.  The monomial label $\lcm(\sigma)$  of each face $\sigma$ of the Taylor complex corresponds to an element of this lattice. This structure encodes the relationships between the monomials in 
$I$, while the Taylor complex provides a geometric representation of these connections.

Although Taylor's construction is often far from
minimal, it produces a resolution of any monomial ideal $I$, giving the
upper bounds $\beta_{i}(I) \leq \binom{q}{i+1}  $ for the Betti
numbers of $I$ where $\binom{q}{i+1}$ is the number of $i$-faces of a
$q$-simplex.  However, if $r$ is a
positive integer, then the number of generators of $I^r$ generally grows
exponentially and as a result, so do the bounds on the Betti numbers
of $I^r$ given by Taylor's resolution.  Consequently, much effort has been invested in using combinatorics to find concrete structures that provide resolutions of $I$, or of $I^r$ for some $r \ge 2$, that are closer to minimal. 

Taylor's work generalizes to simplical complexes and cellular complexes. For definitions and basic background on cellular complexes, see \cite{M,OW}. Let $I$ be minimally generated by monomials
$m_1,\ldots,m_q$ and let $\Delta$ be a simplicial or cellular complex on $q$ vertices
$v_1,\ldots,v_q$. Label each vertex $v_i$ with the monomial $m_i$,
and label each face or cell of $\Delta$ with the least common multiple of
the labels of its vertices. Under certain circumstances the chain complex of $\Delta$ can be homogenized using the monomial labels to give a free resolution of $I$. In this case, we say that $\Delta$ {\bf
  supports a free resolution} of $I$ and the resulting free resolution
is called a {\bf simplicial resolution} or a {\bf cellular resolution} of $I$. We refer the reader to \cite{P} for an in-depth study of this method.

 Consider an ideal $I$ that is generated by monomials $m_1,m_2,m_3,m_4$ satisfying the divisibility relation $m_1 \mid \lcm(m_2, m_3)$. In this case, the labels of $\sigma=\{v_1,v_2,v_3\}$ and of $\sigma'=\{v_2,v_3\}$ are the same, which implies that the Taylor resolution is not minimal. \cref{f:four} depicts the Taylor simplex on four vertices and the subcomplex $\Delta$ obtained by removing the faces containing  $\sigma'$. Both support  free resolutions of $I$. However, the one on the right provides better bounds on the Betti numbers of $I$.
 We provide an example of an ideal $I$ and a simplicial complex $\Delta$ that is strictly smaller than $\Taylor(I)$ 
 that supports a resolution of $I$. 
 
  \begin{example}\label{running_ex}  Let $I = (ab, bcd, aef, cg)$. Using the notation above, $m_1 =ab, m_2 = bcd, m_3 = aef$, and $m_4 =cg$. Notice that $m_1 \mid \lcm(m_2, m_3)$. In this case, the complex $\Delta$ of \cref{f:four} can be seen to support a minimal resolution, so the bounds on the Betti numbers provided on the right side of this figure are attained. 
 \end{example}

\begin{figure}[!htbp]
{\small
$\begin{array}{ccc}
\mbox{Taylor complex of } I 
&&
\Delta \mbox{ supports resolution of } I \\
\begin{tikzpicture}
\tikzstyle{point}=[inner sep=0pt]
\node (a)[point,label=left: $m_2$] at (-0.8,0) {};
\node (b)[point,label=right:$m_1$] at (0,-1){};
\node (c)[point,label=right:$m_4$] at (0.4,1.3){};
\node (d)[point,label=right:$m_3$] at (1,0){};
\node (e)[point] at (0,0){};

\draw  [fill=gray!30] (d.center) -- (c.center) -- (b.center) -- cycle;
\draw  [fill=gray!30] (a.center) -- (c.center) -- (b.center) -- cycle;
\draw[dashed] (a) -- (d);

\filldraw [black] (a.center) circle (1pt);
\filldraw [black] (b.center) circle (1pt);
\filldraw [black] (c.center) circle (1pt);
\filldraw [black] (d.center) circle (1pt);
\end{tikzpicture} 
&&
\begin{tikzpicture}
\tikzstyle{point}=[inner sep=0pt]
\node (a)[point,label=left: $m_2$] at (-0.8,0) {};
\node (b)[point,label=right:$m_1$] at (0,-1){};
\node (c)[point,label=right:$m_4$] at (0.4,1.3){};
\node (d)[point,label=right:$m_3$] at (1,0){};
\node (e)[point] at (0,0){};

\draw  [fill=gray!30] (d.center) -- (c.center) -- (b.center) -- cycle;
\draw  [fill=gray!30] (a.center) -- (c.center) -- (b.center) -- cycle;

\filldraw [black] (a.center) circle (1pt);
\filldraw [black] (b.center) circle (1pt);
\filldraw [black] (c.center) circle (1pt);
\filldraw [black] (d.center) circle (1pt);
\end{tikzpicture}\\ 
&&\\
(\beta_0(I),\beta_1(I),\beta_2(I),\beta_3(I)) \leq (4,6,4,1)&&
(\beta_0(I),\beta_1(I),\beta_2(I))\leq (4,5,2)
\end{array}$
}
\caption{}\label{f:four}
\end{figure}

Given an ideal $I$ generated by $q$ square-free monomials, the authors in \cite{Lr} introduced a simplicial complex $\LL^r_q$ that supports a resolution of $I^r$. Except when $r=1$, in which case $\LL^1_q$ coincides with the simplex on $q$ vertices, the complex $\LL^r_q$ is considerably smaller than a simplex on $\binom{q+r-1}{r-1}$ vertices, which is the maximal number of generators of $I^r$. In order to precisely define $\LL^2_q$, we first set some notation that will be used throughout the paper.

\begin{notation}\label{n:setup}
For $q$ and $r$ positive integers, $\ba=(a_1,\ldots,a_q) \in \NN^q$, and $A \subseteq [q],$  we use the following notation, assumptions, and conventions:
 \begin{itemize} 
 \item $\be_i$  is the $i^{th}$ standard unit vector in $\RR^q$;
 \item $\Nrq =\{ a_1\be_1+\cdots + a_q\be_q  \st \sum_{i=1}^q a_i = r \}$;
    \item We write $\be_{ij} = \be_i + \be_j$ where $\be_{ii}=\be_i+\be_i$.
    \end{itemize}
\end{notation}

\begin{definition}[{\bf The simplicial complex $\LL_q^2$}~{\cite[Definition 1]{L2}}, {\cite[Proposition 4.3]{Lr}}]\label{d:L2}  

For an integer $q \geq 3$, the simplicial complex $\LL^2_q=\langle \mathcal B, G_1,\ldots,G_q \rangle$ has vertex set $\N^2_q$ and facets: 
$$
\mathcal B=\{\be_{ij} \st 1 \leq i< j \leq q \}
\qand
G_i= \{\be_{ij} \st 1 \leq j \leq q \} 
\qfor i \in [q].
$$

For $q=1$ and $q=2$ we use the same construction but note that $\mathcal{B}$ is empty for $q=1$ and is a face but not a facet of $\LL^2_q$ for $q=2$. 

When using a labeling on $\LL_q^2$ for a specific monomial ideal $I=(m_1, \ldots, ,m_q)$, we use the labeling induced from that of the Taylor complex. That is, $\be_{ij}$ is labeled by $m_im_j$. 
\end{definition}

\begin{example} \label{LL4}
We illustrate $\LL^2_4$ below in Figure \ref{f:L24}. Note that there are five facets; one simplex with six vertices in the center shaded gray and four tetrahedrons with the vertices $m_i^2, m_im_j, m_im_k, m_im_l$ for $i, j, k, l$ distinct elements in $\{1,2,3,4\}$.

\begin{figure}[!htbp]
\begin{center}
\tdplotsetmaincoords{60}{150}
\begin{tikzpicture}[tdplot_main_coords]

    \coordinate (A) at (0,0,0);
    \coordinate (B) at (1.75,0,0);
    \coordinate (C) at (3.5,0,0);
    \coordinate (D) at (0,1.75,0);
    \coordinate (E) at (0,3.5,0);
     \coordinate (F) at (1.75,1.75,0);
      \coordinate (G) at (0,0,1.75);
      \coordinate (H) at (0,0,3.5);
      \coordinate (I) at (1.75,0,1.75);
      \coordinate (J) at (0,1.75,1.75);
\draw [fill=gray!10] (B.center) -- (I.center) -- (G.center) -- (J.center)-- (D.center) -- (F.center);
      
  \draw (F) -- (B) ;
  \draw (F) -- (D);
   \draw [thick] (A) -- (B);
      \draw [thick] (A) -- (D);   
    \draw [thick] (A) -- (G);
       \draw (G) -- (H);
          \draw (G) -- (J);
          \draw (G) -- (I);
          \draw (B) -- (I);
          \draw (G) -- (B);
          \draw (B) -- (C);
          \draw (I) -- (C);
\draw (H) -- (I);
\draw (H) -- (J); 
\draw (J) -- (D);  
\draw (E) -- (J);
\draw (F) -- (J);
\draw (G) -- (F);
\draw (I) -- (F);
\draw (G) -- (D);

\draw (I) -- (J);
\draw (C) -- (F);
\draw (F) -- (E);
\draw (B) -- (D);
\draw (I) -- (D);
\draw (B) -- (J);

    \foreach \point in {A, B, C, D, E, F, G, H, I, J} \fill (\point) circle (2pt);

    \node[above right] at (A) {\tiny{$m_1^2$}};
    \node[above left] at (B) {\tiny{$m_1m_2$}};
    \node[above left] at (C) {\tiny{$m_2^2$}};
    \node[right] [label distance=-6pt]  at (D) {\tiny{$m_1m_3\ \ $}};
    \node[above right] at (E) {\tiny{$m_3^2$}};
    \node[ below] at (F) {\tiny{$m_2m_3$}};
     \node[above left] at (G) {\tiny{$m_1m_4$}};
    \node[right] at (H) {\tiny{$m_4^2$}};
    \node[left] at (I) {\tiny{$m_2m_4$}};
     \node[right] at (J) {\tiny{$m_3m_4$}};


    \draw[-] (0,0,0) -- (3.5,0,0) node[anchor=north east]{};
    \draw[-] (0,0,0) -- (0,3.5,0) node[anchor=north west]{};
    \draw[-] (0,0,0) -- (0,0,3.5) node[anchor=south]{};
\end{tikzpicture} 
\end{center}
    \caption{The labeled simplicial complex $\LL_4^2$}
    \label{f:L24}    
\end{figure}
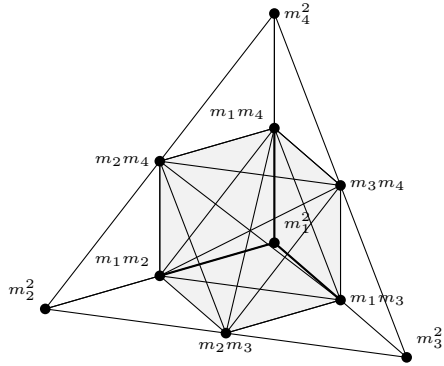

\end{example}

If a single divisibility condition exists among the generators of $I$, then additional faces may be removed from $\Taylor(I)$ and $\LL_q^2$, in the cases of $r=1, 2$ respectively, to get a complex that supports a resolution of $I^r$. This will be proved in \cref{s:morse matchings}.  The removal of faces uses discrete Morse theory, which will be introduced in \cref{s:Morse_basics}, and relies on identifying faces that share the same lcm labels. Since a face and one of its subfaces have the same lcm label precisely when the monomial labels of the additional vertices already divide the lcm of the monomial labels of the subface, we need to understand divisibility relations among the monomial generators of $I$ and $I^2$.  We fix an order on the monomial generators of $I$ and encode divisibility relations  using the fixed indices. For example, in \cref{running_ex} we write $(1,\{2,3\})$ for $m_1 \mid \lcm(m_2,m_3)$. This notation was introduced in \cite{Divisibilities}  and is made precise below.

\begin{definition}[{\bf Divisibility relations}]
\label{d: div-rel}
Let $U=\{u_i\st i\in \Lambda\}$ denote a set   of distinct monomials indexed by a finite set $\Lambda$, so that $u_i=u_j$ implies $i=j$. We define a {\bf divisibility relation on} $U$, encoded as the pair $(b,B)\in \Lambda\times 2^\Lambda$, to be a relation of the form 
\begin{equation}\label{e:div-rel}
u_b\mid \lcm(u_i\st i\in B) \qforsome  b\in \Lambda \qand  \emptyset \ne B \subseteq \Lambda.
\end{equation} 
A relation $(b,B)$ is said to be {\bf trivial} if $b\in B$, and we say that $(c,C)$ is an {\bf extension} of $(b,B)$ if $b=c$ and $B \subseteq C$. In addition, we say a divisibility relation $(b,B)$ on $U$ is {\bf minimal} if it is non-trivial and is not an extension of any other divisibility relation on $U$.

When working with a monomial ideal $I$, by {\bf divisibility relations on $I$} we mean divisibility relations on the unique set of minimal monomial generators of $I$. 
\end{definition}

The set $\R(I)$ consisting of all divisibility relations on a monomial ideal $I$ contains the trivial relations on $I$ and is closed under extensions. Thus, in order to describe this set, it suffices to describe the set of minimal divisibility relations on $I$ since all other relations are either trivial or extensions of a minimal relation. 

To study ideals whose generators satisfy divisibility relations, we use a special class of ideals defined in \cite{Divisibilities}, called $\D$-extremal ideals. These ideals are extremal relative to a set of divisibility relations $\D$ in the sense that the ideals are as general as possible among square-free monomial ideals satisfying the divisibility relations in $\D$. To define these ideals, we start with a set of desired relations indexed by subscripts as above, and then define generators of $\ED$ in such a way that they will satisfy these relations. We recall the precise definition of $\ED$ below.

\begin{definition}[{\bf $\D$-extremal ideals}]\label{d:D-def} 
Let $q>0$ and $d\geq0$.  Define the polynomial ring $S_{[q]}$ by 
\begin{equation*}\label{e:Sq}
S_{[q]}=\sfk[y_A\st \emptyset\ne A \subseteq [q]]
\end{equation*}
and set 
$$
\D=\{(b_1,B_1),\ldots,(b_d,B_d)\} \subseteq [q]\times (2^{[q]}\ssm\{\emptyset\}) 
$$
where $\D = \emptyset$ if $d=0$.
Define 
\begin{align*}
Q(\D)&=\{A\subseteq [q] \, \st A\neq \emptyset, \text{ and for all } j\in [d],\   b_j\notin A \text{ or } A\cap B_j\ne\emptyset\}.
\end{align*}
For $i \in [q]$, we  define square-free monomials 
\[
\ed{i}={\displaystyle \prod_{\substack{A\in Q(\D)\\ i\in A}}y_A}\,.
\]
Given $\ba=(a_1, \dots, a_q)\in \Nrq$, we set 
\begin{equation}\label{eda}
\ped^\ba=\ed{1}^{a_1}\cdots \ed{q}^{a_q}.
\end{equation}
The {\bf $\D$-extremal ideal} is defined as the square-free monomial ideal in the polynomial ring $S_{[q]}$ generated by the $\ed{i}$, namely
$$
\ED=(\ed{1}, \ \ldots , \ \ed{q}).
$$
\end{definition}
Define $\E_q=\E_{q,\emptyset}$ and note that $Q(\emptyset)=\{A\subseteq [q] \, \st A\neq \emptyset\}.$ This special case was introduced in \cite{Lr} and has been studied due to its extremal properties (see for instance \cite{Scarf, extremal}).

\begin{Remark}\label{remark:generators of extremals}
When $|B_i| \geq 2$ for each $i$, $\ED$ is minimally generated  by $\ed{1}, \ \ldots , \ \ed{q}$ and $\EDr$ is minimally generated by $\{\ped^\ba \st \ba \in \Nrq\}$ by \cite[Proposition 4.1]{Divisibilities}.
\end{Remark}

For the remainder of this paper, the sets $\D$ are formed by divisibility relations on minimal generators of a square-free monomial ideal $I$, and thus we may always assume $|B_i| \ge 2$. In particular, we assume
\begin{equation}\label{e:oneOK}
\D=\{(b_1,B_1),\ldots,(b_d,B_d)\} \subseteq [q]\times (2^{[q]}\ssm\{\emptyset\}) \qwhere |B_i| \ge 2 \qforall i\in [d]. 
\end{equation}
We will also simplify notation by writing $y_{_{1}}$ in place of $y_{\{1\}}$, $y_{_{12}}$ in place of $y_{\{1,2\}}$, etc.

\begin{example} \label{mainex} 
When $q=4$, set $\D = \{(1,\{2,3\})\}$. Then $d=1$, $B = \{2,3\}$, and $b_1 = 1$. Since the set $Q(\D)$ consists of all non-empty subsets of $[4]$ except $\{1\}$ and $\{1, 4\}$, and thus the variables $y_{_{1}}$ and $y_{_{14}}$ do not appear in any generator of the ideal, the ideal $\E_{4,\D}$ in $S_{[4]}$ is generated by the monomials
$$
\begin{array}{llll}
&\ed{1}=y_{_{12}}y_{_{13}}y_{_{123}}y_{_{124}}y_{_{134}}y_{_{1234}}&
&\ed{2}=y_{_{2}}y_{_{12}}y_{_{23}}y_{_{24}}y_{_{123}}y_{_{124}}y_{_{234}}y_{_{1234}}\\
&\ed{3}=y_{_{3}}y_{_{13}}y_{_{23}}y_{_{34}}y_{_{123}}y_{_{134}}y_{_{234}}y_{_{1234}}&&
\ed{4}=y_{_{4}}y_{_{24}}y_{_{34}}y_{_{124}}y_{_{134}}y_{_{234}}y_{_{1234}}.
\end{array}
$$
Note that $\ed{1} \mid \lcm (\ed{2}, \ed{3})$.  
\end{example}

\section{\bf The divisibility relations of  ${\EDsq}$}\label{sec:ED2}

 In this section, we discuss divisibility relations on $I^2$, where $I$ is an ideal minimally generated by $q\ge 1$ square-free monomials. More precisely, if $I$ satisfies a given divisibility relation, we identify  divisibility relations that are satisfied by $I^2$. These relations will be instrumental in creating a Morse matching on $\LL_q^2$ in \cref{s:morse matchings}. 

When $I=\ED$ for some set of divisibility relations $\D$ as in \eqref{e:oneOK}, we know that the minimal monomial generators of $\ED$ satisfy the relations in $\D$ and do not satisfy additional relations other than those which can be deduced from $\D$, as seen in~\cite[Theorem~3.4]{Divisibilities}. In addition, by~\cite[Corollary~4.3]{Divisibilities} if $I$ is any ideal minimally generated by $q$ square-free monomials satisfying the relations in $\D$, then $I^2$ satisfies all relations on $\EDsq$. Thus we focus on relations satisfied by $\EDsq$. We consider two  cases: when $\D$ is empty and when $\D$ consists of one relation, in which case we may assume
\begin{equation}
\label{e:onerel}
\D=\{(1,\{2,\dots, s\})\},
\quad \mbox{which means} \quad 
 \ed{1} \mid \lcm(\ed{2}, \ldots, \ed{s}),
 \end{equation}
 where $s\geq 3$. In these two cases we have a full understanding of the divisibility relations on $\ED$. More precisely, by \cite[Theorem~3.4]{Divisibilities}, we see that when $\D=\emptyset$ there are no minimal divisibility relations on  $\ED$ and when $\D$ is as in \eqref{e:onerel}, the relation $(1,\{2,\dots, s\})$ is the only minimal divisibility relation on $\ED$. On the other hand, there are many more minimal divisibility relations on $\EDsq$. 
 In \cref{p:lcm-general} we describe a complete set of minimal divisibility relations on  $\EDsq$ in the two cases of interest.   As a first step, in \cref{p:lcm-general-new} we establish some divisibility relations that hold on $I^2$ for any square-free monomial ideal $I$, and establish two classes of additional divisibility relations on $I^2$ under the assumption that $I$ satisfies the relation \eqref{e:onerel}.

\begin{proposition}[{\bf Divisibility relations on $I^2$}]\label{p:lcm-general-new}
Let $q>0$ and let $I$ be a square-free monomial ideal minimally generated by
$m_1,\ldots,m_q$. Then $I^2$ satisfies the divisibility relations
 \begin{equation}
 \label{e:L2}
 m_am_b \mid \lcm (m_a^2,m_bm_c)  \qwhere   a,b,c\in [q]\,. 
\end{equation}
If $I$ satisfies the divisibility relation
\begin{equation}
\label{e:lcm0}
m_1 \mid \lcm(m_2,\ldots,m_s) 
\end{equation} 
where $3\le s\le q$, then $I^2$ satisfies the divisibility relation
\begin{equation}
\label{e:m2}
m_1m_j \mid \lcm(m_2m_{t_2},\ldots, m_sm_{t_s})
\end{equation}
where $j\in [s]$ and either $t_j=j\ge 2$ or $t_k \in \{1,j,k\}$ for all $2 \leq k \leq s$, and the divisibility relation
\begin{equation}
\label{e:m22}
m_1m_j \mid \lcm(m_2m_{t_2},\ldots, m_sm_{t_s}, m_u m_j)
\end{equation}
where $s<j \le q$ and either $u=j$ or
$t_k\in \{1,j,k\}$ for all $2\le k\le s$.
\end{proposition}

 \begin{proof} The relation in~\eqref{e:L2} appears implicitly in~\cite{L2}. We repeat the proof for completeness. Let $x$ be a variable such that $x\mid m_am_b$. Then $x\mid m_a$ or $x\mid m_b$, implying that $x$ divides $\lcm(m_a^2,m_bm_c)$.  If $x^2\mid m_am_b$, then since $m_a$ and $m_b$ are square-free, we must have $x\mid m_a$ and $x\mid m_b$, which implies that $x^2\mid m_a^2$ and hence $x^2$ divides $\lcm(m_a^2,m_bm_c)$. 

Assume \eqref{e:lcm0} holds and $x$ is a variable such that $x\mid m_1m_j$, and hence $x\mid m_1$ or $x\mid m_j$. If $x\mid m_1$, then by \eqref{e:lcm0}, we have $x \mid m_k$ for some $2\leq k \leq s$ and thus $x$ divides the lcms in \eqref{e:m2} and \eqref{e:m22}.
If $x\mid m_j$, then either $j\le s$ and $x$ divides a term of the lcm in \eqref{e:m2}, or $j>s$ and $x$ divides the last term of the lcm in \eqref{e:m22}. 

Assume now that $x^2\mid m_1m_j$. Since $m_1$ and $m_j$ are square-free, we have $x\mid m_1$, and hence $x\mid m_k$ for some $2\le k\le s$, and $x\mid m_j$. It follows that 
\[
x^2\mid m_km_1, \quad x^2\mid m_k^2, \quad x^2\mid m_km_j\quad\text{and}\quad x^2\mid m_j^2\,.
\] 
Observing that in all situations one of the monomials $m_km_1$, $m_k^2$, $m_km_j$ or $m_j^2$ is among the monomials in the lcm of \eqref{e:m2} and \eqref{e:m22}, we conclude that $x^2$ divides the lcm in either case. 
\end{proof}

We now turn our attention to the extremal ideals $\ED$, with $\D$ as in \eqref{e:oneOK}. In this case, we
take $m_i=\ed{i}$ in \cref{p:lcm-general-new}.  For convenience, we adopt \cref{n:setup}, and describe the generators of $\EDr$ using elements of $\N^r_q$, so that $\ED$ is generated by $\ped^{\be_i}=\ed{i}$
with $i\in [q]$,  and $\EDsq$ is generated by  $\ped^{\be_{ij}}=\ed{i}\ed{j}$ with $i,j\in [q]$, where $\be_i\in \N_q^1$ and $\be_{ij}\in\N_q^2$ are as defined in \cref{n:setup}.  When $\sigma \subseteq \N_q^1$ and   $\tau \subseteq \N_q^2$  
we denote
$$
\lcm(\sigma)= \lcm(\ed{a}\st \be_a\in \sigma)
\qand 
\lcm(\tau)=\lcm \big ( \ed{a}\ed{b} \st \be_{ab}\in \tau \big ).
$$ 
 For example, with this notation, we have \[
\lcm(\{\be_{11}, \be_{12}, \be_{13}\})=\lcm(\ed{1}^2,\, \ed{1}\ed{2},\, \ed{1}\ed{3}).
\]

Before investigating the divisibility relations of $\EDsq$, we discuss some preliminaries. 

\begin{lemma}\label{l:A-new} 
  Let $q>0$ and let $\D$ be as in \eqref{e:oneOK}. 
  Suppose  $\tau\subseteq \N^2_q$, $\gamma\subseteq \N^1_q$,  $A\in Q(\D)$.  Then, for $i,v\in A$ and $w\in [q]$, we have:  
  
\begin{enumerate}  
\item \label{eq:A1}
 $\ed{i}\mid \lcm(\gamma) \implies \be_a\in \gamma$\,  for some $a\in A$;
\item \label{eq:A2}
$\ed{i}\ed{v}\mid \lcm(\tau) \implies \be_{ab} \in \tau$\, for some $a,b\in A$;
\item \label{eq:A3}
 $\ed{i}\ed{w}\mid \lcm(\tau) \implies \be_{ab} \in \tau$\, for some $a\in A$, $b\in [q]$.
 \end{enumerate}
 
\end{lemma} 

\begin{proof}
Since $i,v\in A$ and $A \in Q(\D)$, by the definition of $\ed{i}$ and $\ed{v}$,  we must have $y_A \mid \ed{i}$ and $y_A \mid \ed{v}$.

To show item (1), suppose $\ed{i}\mid \lcm(\gamma)$. Then 
 $$y_A \mid \ed{i} 
 \Longrightarrow
 y_A \mid \lcm(\gamma)
 \Longrightarrow
 y_A \mid \ed{a} \qforsome \be_a \in \gamma.$$ By the definition of $\ed{a}$, we have $a\in A$.

To see item (2), suppose $\ed{i}\ed{v}\mid \lcm(\tau)$.  Then since $y_A \mid \ed{i}$ and  $y_A \mid \ed{v}$
 $$ 
 (y_A)^2 \mid \ed{i}\ed{j}
 \Longrightarrow 
 (y_A)^2 \mid \lcm(\tau)
 \Longrightarrow (y_A)^2\mid \ed{a}\ed{b} \qforsome \be_{ab}\in \tau.
 $$
 Therefore, since $\ed{a}$ and $\ed{b}$ are both square-free,  $y_A\mid \ed{a}$  and $y_A\mid \ed{b}$, which by definition means $a,b\in A$. 
 
Finally, for item (3) suppose $\ed{i}\ed{w}\mid \lcm(\tau)$.  Now 
 $$y_A \mid \ed{i} 
 \Longrightarrow 
 y_A \mid \ed{i}\ed{w}
 \Longrightarrow 
 y_A \mid \lcm(\tau)
 \Longrightarrow y_A\mid \ed{a}\ed{b} \qforsome \be_{ab}\in \tau.
 $$
 Therefore $y_A\mid \ed{a}$  or $y_A\mid \ed{b}$, which means $a\in A$ or $b\in A$. 
\end{proof}

Before considering further the divisibilities on $\EDsq$, we state a result for the first power, in the special case where all the sets $B_i$ of $\D$ are equal.   This result can also be deduced from ~\cite[Theorem~3.4]{Divisibilities}, but doing so would require the introduction of additional notation that is not needed elsewhere. Instead we give a short proof below.  

\begin{lemma}
\label{l:specialD}
Assume $q\ge s\ge 3$, $B\subseteq [q]$ with $|B| \ge 2$, and  $\D\subseteq [q]\times 2^{[q]}$ is given by 
$$\D =\{(b,B)\colon b\in J\} \qforsome J\subseteq [q]\ssm B.$$
Assume $\sigma\subseteq \N^1_q$ and $\be_i\notin\sigma$. Then 
$\ed{i}\mid \lcm(\sigma)$
if and only if $i\in J$ and $\{\be_j\colon j\in B\}\subseteq \sigma$. 
\end{lemma}
\begin{proof}
  Assume  $\ed{i}\mid \lcm(\sigma)$. Let $j\in B$. Then the set $A=\{i,j\}$ is in $Q(\D)$ but $\be_i \not\in \sigma$, and hence $\be_j\in \sigma$ by \cref{l:A-new}(\ref{eq:A1}). This shows  $\{\be_j\colon j\in B\}\subseteq \sigma$. If $i\notin J$, then $A=\{i\}$ is also in $Q(\D)$, and  \cref{l:A-new}(\ref{eq:A1}) implies $\be_i\in \sigma$, contradicting our assumption. The converse follows from the fact that the minimal monomial generators of $\ED$ satisfy the relations in $\D$. 
\end{proof}

The next result considers the most basic case where $\D=\emptyset$ (and hence $\ED=\E_q)$, for which we give a complete characterization of all the divisibility relations on $\EDsq$. Note that conditions (1) and (2) below are mutually exclusive when the inclusions therein are equalities. 

\begin{proposition}
\label{p:needed-div-empty}
Assume $\D=\emptyset$, $1 \leq i\le j\le q$,  $\sigma\subseteq \N^2_q$ and $\be_{ij} \notin \sigma$. Then 
\[
\ed{i}\ed{j}\mid \lcm(\sigma)
\]
if and only if there exist integers $a, b \in [q]$ such that one of the following holds: 
\begin{enumerate}
\item \label{i:n1e} $\{\be_{jj}, \be_{ia}\}\subseteq \sigma$, with $i\ne j$ and $a\notin \{i,j\}$;
\item \label{i:n2e} $\{\be_{ii}, \be_{jb}\}\subseteq \sigma$, with $i\ne j$ and $b\ne i$. 
\end{enumerate}
\end{proposition}

\begin{proof}
Assume $\ed{i}\ed{j}\mid \lcm(\sigma).$  Assume $i=j$. Then applying \cref{l:A-new}~(\ref{eq:A2}) with $A=\{i\}$ gives $\be_{ii}\in \sigma$, a contradiction. Hence $i\ne j$.

Using \cref{l:A-new}~(\ref{eq:A2}) with $A=\{i,j\}$, we see that $\be_{ii}\in \sigma$ or $\be_{jj}\in \sigma$ since $\be_{ij} \notin \sigma$. If $\be_{ii}\in \sigma$, then we apply \cref{l:A-new}~(\ref{eq:A3}) with $A=\{j\}$ to conclude that $\be_{jb}\in \sigma$ for some $b\in [q]$, where $b\ne i$ since $\be_{ij}\notin\sigma$. Thus (\ref{i:n2e}) holds. If $\be_{jj}\in \sigma$, then we apply \cref{l:A-new}~(\ref{eq:A3}) with $A=\{i\}$ to conclude that $\be_{ia}\in \sigma$ for some $a\in [q]$, where $a\ne j$ since $\be_{ij}\notin\sigma$. If $a\ne i$, then (\ref{i:n1e}) holds.  If $a=i$, then (\ref{i:n2e}) holds. 

 The converse statement follows from \eqref{e:L2} by noting that when $I=\ED$ the generators are $m_i=\ed{i}$ for $1 \le i \le q$.
\end{proof}

We now consider the case when $\D$ consists of a single relation. More precisely, we assume
\begin{align}\label{eq:Ae} 
&\D = \{(1, \{2, \ldots, s\}) \big \}  \qwhere 3 \leq s \leq q,  \\ \nonumber
&Q(\D) =
\{ A \subseteq [q] \st A \neq \emptyset, \mbox{ and either } 1\notin A \mbox{ or }  A\cap \{2,3,\dots, s\}\ne \emptyset \},
\\ \nonumber
&\ED \quad \mbox{ is generated by } \quad \ed{i}={\displaystyle \prod_{\substack{ A\in Q(\D)\\ i\in A}}y_A}\qfor i\in[q].
\end{align}

The next lemma gives a key technical ingredient for our results. 

\begin{lemma}\label{l:use}
Assume the setting of \eqref{eq:Ae}. Let $\sigma\subseteq \N^2_q$, $2\le k \le s$, and  $1\le i\le j\le q$ be such that 
$$\ed{i}\ed{j}\mid \lcm(\sigma\smallsetminus \{\be_{ij}\}).$$ 
 The following then hold: 
\begin{enumerate}
\item\label{i:o} If $i=1$, then there exists $t\in [q]$ such that $\{\be_{1t},\be_{kt}\}\cap(\sigma\smallsetminus\{\be_{1j}\})\ne \emptyset$.
\item\label{i:a} If $i=1$ and $2\le j\le s$,  then $\{\be_{11},\be_{jj}\}\cap\sigma\ne\emptyset$. 
\item \label{i:b}If $i=1=j$, then $\{\be_{1k},\be_{kk}\}\cap\sigma\ne\emptyset$.
\item \label{i:c} If $i=1$ and $j>s$, then 
$\{\be_{11},\be_{kk},\be_{jj},\be_{1k},\be_{jk}\}\cap \sigma\ne\emptyset$. 
\item \label{i:d} If $i>1$, then $i\ne j$ and $\{\be_{ii},\be_{jj}\}\cap\sigma\ne\emptyset$. 
\item \label{i:e1}  If $i>1$, then there exists $u\in [q]\smallsetminus\{j\}$ with $\be_{iu}\in \sigma$.
\item \label{i:e2}  If $j>1$, then there exists $u\in [q]\smallsetminus\{i\}$ with $\be_{ju}\in \sigma$.
\end{enumerate}
\end{lemma}

\begin{proof}
 \eqref{i:o} The set $A=\{1,k\}$ is in $Q(\D)$. The conclusion follows from \cref{l:A-new}~(\ref{eq:A3}), with $\tau = \sigma\smallsetminus\{\be_{1j}\}$, $a \in \{1,k\}$, and $b = t$.

\eqref{i:a} The set  $A=\{1,j\}$ is in $Q(\D)$. Using \cref{l:A-new}~(\ref{eq:A2}) with $\tau = \sigma\smallsetminus\{\be_{1j}\}$, it follows that $\be_{ab}\in\sigma\smallsetminus \{\be_{1j}\}$ for some  $a,b\in \{1,j\}$, and the only possible options are $a=b=1$ or $a=b=j$. 

\eqref{i:b} The set $A=\{1,k\}$ is in $Q(\D)$. The hypothesis and \cref{l:A-new}~(\ref{eq:A2}) with $\tau = \sigma\smallsetminus\{\be_{11}\}$, imply that $\be_{ab}\in \sigma\smallsetminus \{\be_{11}\}$ for some $a,b\in \{1,k\}$.   The only possible options are $a=1$ and $b=k$ or $a=b=k$. 

\eqref{i:c} The set $A=\{1,k,j\}$ is in $Q(\D)$. The hypothesis and \cref{l:A-new}~(\ref{eq:A2}) show that $\be_{ab}\in \sigma\smallsetminus \{\be_{1j}\}$ for some $a,b\in \{1,k,j\}$.  

\eqref{i:d} Since $i\ne 1$, the set $A=\{i\}$ is in $Q(\D)$. If $j=i$, then the hypothesis and \cref{l:A-new}~(\ref{eq:A2}) imply $\be_{ii}\in \sigma\smallsetminus \{\be_{ii}\}$, a contradiction. 
Therefore, $1<i<j$ and the set $A=\{i,j\}$ is in $Q(\D)$. The result follows as in \eqref{i:a}.

\eqref{i:e1} and \eqref{i:e2}. Let $\{v,w\}=\{i,j\}$, $A=\{v\}$ with $v>1$. Then $A\in Q(\D)$.  \cref{l:A-new}~(\ref{eq:A3}) implies that $\be_{vu}\in \sigma\smallsetminus \{\be_{ij}\}$ for some $u\in [q]$. Since $\be_{uv}\ne \be_{ij}$, it follows that $u\ne w$. 
\end{proof}

We now give a complete characterization of the divisibility relations of $\EDsq$ in terms of the faces of the Taylor complex. 

\begin{proposition}
\label{p:needed-div-T}
Assume the setting of \eqref{eq:Ae}. Let $\sigma\subseteq \N^2_q$, $1 \leq i\le j\le q$, and $\be_{ij} \notin \sigma$.  Then 
\begin{equation}
\label{e:div}
\ed{i}\ed{j}\mid \lcm(\sigma)
\end{equation}
if and only if there exist integers $t_2, \dots, t_s, a, b, u\in [q]$ such that one of the following holds: 
\begin{enumerate}
\item \label{i:n1} $\{\be_{jj}, \be_{ia}\}\subseteq \sigma$, \,  $i\ne j$,\, and \,$a\notin \{i,j\}$; 
\item \label{i:n2} $\{\be_{ii}, \be_{jb}\}\subseteq \sigma$, \, $i\ne j$,\, and \,$b\ne i$;
\item\label{i:n3} $\{\be_{2t_2},\dots, \be_{st_s}\}\subseteq \sigma$,\, $i=1$,\, $t_k\in \{1,j,k\}$\, for all \, $2\le k\le s$,  and 
\begin{enumerate}
  \item \label{i:n3a}$j=1$; \, or 
    \item \label{i:n3b}$j>s$ \, and \, $j\in \{t_2, \dots, t_s\}$; 
\end{enumerate}
\item\label{i:n4} $\{\be_{2t_2},\dots, \be_{st_s}, \be_{u j}\}\subseteq \sigma$,\, $i=1$\, and 
\begin{enumerate}
\item \label{i:n4a} $j>s$, $u>s$, $j\ne u$
     and $t_k\in \{1,k\}$  for all $2\le k\le s$; or
\item\label{i:n4b}  $j=u>1$ and $t_k> 1$ for all $2\le k\le s$. 
\end{enumerate}
\end{enumerate} 
\end{proposition}

\begin{proof}
Assume $\ed{i}\ed{j}\mid \lcm(\sigma).$

If $i>1$, then \cref{l:use}~(\ref{i:d}) shows that $i\ne j$ and $\be_{ii}\in \sigma$ or $\be_{jj}\in \sigma$ since $\be_{ij} \notin \sigma$. Also, note that \cref{l:use}~(\ref{i:e1}) implies that there exists an element $a\in [q]\smallsetminus\{j\}$ such that $\be_{ia} \in \sigma$.  Likewise, since $j \geq i$, \cref{l:use}~(\ref{i:e2}) implies that there exists $b\in [q]\smallsetminus\{i\}$ such that $\be_{jb} \in \sigma$.  In particular,  \eqref{i:n1} or \eqref{i:n2} holds.

Assume $i=1$ and $j=1$. Then \cref{l:use}~(\ref{i:b}) implies that $\be_{1k}\in \sigma$ or $\be_{kk}\in \sigma$ for all $k$ with $2\le k\le s$. Then (\ref{i:n3a}) holds with 
\[
t_k=\begin{cases}
1 &\qif  \be_{1k}\in \sigma\\
k &\quad \text{otherwise}.
\end{cases}
\]

Assume $i=1$ and $j>1$. \cref{l:use}~(\ref{i:e2}) implies there exists $u\in [q]\smallsetminus\{1\}$ such that $\be_{u j}\in \sigma$. If $\be_{11}\in\sigma$, then \eqref{i:n2} holds. Assume now $\be_{11}\notin \sigma$. 

Suppose $\be_{jj}\in \sigma$. In this case, we may assume $u=j$. If $\be_{1a}\in \sigma$ for some $a\ne j$, then \eqref{i:n1} holds since $a \neq 1$.  Thus, assume that $\be_{1a}\notin \sigma$ for all $a\ne j$. We will show that (\ref{i:n4b}) holds. Recall that $\be_{1j} \not\in \sigma$ by hypothesis. Then \cref{l:use}\eqref{i:o} implies that for each $k\in \{2, \dots,s\}$ there exists $t_k\in [q]$ such that $\be_{kt_k}\in \sigma$. Since $t_k\ne 1$ for all $2 \le k \le s$,  we conclude that (\ref{i:n4b}) holds. 

Suppose $\be_{jj}\notin \sigma$. Since $\be_{11}\notin \sigma$ as well, \cref{l:use}\eqref{i:a} shows that we must have $j>s$.  By \cref{l:use}~(\ref{i:c}), for each $k\in \{2,\dots, s\}$ there exist  $a_k,b_k\in \{1,j,k\}$ such that $\be_{a_kb_k}\in \sigma$.  Since $\be_{11}$, $\be_{jj}$ and $\be_{1j}$ are not elements of $\sigma$, we must have  $\be_{a_kb_k}\in \{\be_{kk}, \be_{1k}, \be_{jk}\}$ for all $k\in \{2,\dots, s\}$.
Hence, for all such $k$, we have  $\be_{kt_k}\in \sigma$ for some $t_k\in \{1,j,k\}$.  Furthermore, if  $u\le s$, since we know $\be_{uj}\in\sigma$, we choose $t_u=j$. 

 If  $j\in \{t_2, \dots, t_s\}$, then (\ref{i:n3b}) holds.  Assume now $u>s$ and $j\notin \{t_2, \dots, t_s\}$. In particular, $t_k\in \{1,k\}$ for all $k$.  Since $\be_{jj}\notin \sigma$, we have $j\ne u$, and hence (\ref{i:n4a}) holds.

 The converse follows directly from \cref{p:lcm-general-new} with $m_t=\ed{t}$ for all $t\in [q]$. Indeed, if (\ref{i:n1}) or (\ref{i:n2}) holds, then \eqref{e:L2} holds, yielding the desired divisibility relation. Otherwise, by the assumptions of \eqref{eq:Ae} we have the divisibility relation \eqref{e:lcm0}. In case (\ref{i:n3}) we have \eqref{e:m2} holds when $j\le s$. When $j>s$, then by assumption $j \in \{t_2, \ldots, t_s\}$. Let $k$ be such that $j=t_k$. Then \eqref{e:m22} holds with $u=k$. 
In case (\ref{i:n4}) we have \eqref{e:m22} holds and the divisibility $\ed{i}\ed{j}\mid \lcm(\sigma)$ follows. 
\end{proof}

\begin{remark}\label{r:relns exist}
Note that the proof that each of the conditions (1)-(4) in \cref{p:needed-div-T} implies the divisibility relation \eqref{e:div} does not require the full assumptions of \eqref{eq:Ae}. Only the inclusion $\{(1, \{2, \ldots,s\})\} \subseteq \D$ is used in the proof of existence in the final paragraph above.
\end{remark}

In \cref{s:morse matchings}, we will construct a Morse matching on the complex $\LL_q^2$ of \cref{d:L2}, rather than the Taylor complex, and hence we give a version of the previous result that only uses faces of the latter complex. 

\begin{corollary}
\label{p:needed-div}
Assume the setting of \eqref{eq:Ae}. If $1 \leq i\le j\le q$, $\be_{ij} \notin \sigma$, and $\sigma\cup\{\be_{ij}\}$  is a face of $\mathbb L^2_q$, then
\[
\ed{i}\ed{j}\mid \lcm(\sigma)\,,
\]
if and only if $i=1$ and one of the following holds: 
\begin{enumerate}[label=(\roman*)]
\item \label{i}$\{\be_{2j},\be_{3j}, \dots, \be_{sj}\}\subseteq \sigma$; 
\item  \label{ii} $\{\be_{2t_2},\dots, \be_{st_s}\}\subseteq \sigma$ for some $t_2, \dots, t_s$ with $\{t_2, \dots, t_s\}=
\{1,j\}$, and $j>s$;
\item \label{iii} $\{\be_{21},\be_{31},\dots, \be_{s1}, \be_{ju}\}\subseteq \sigma$ for some $u$ with $j\ne u>s$, and $j>s$;   
\end{enumerate}
\end{corollary}
\begin{proof}
Assume $\ed{i}\ed{j}\mid \lcm(\sigma)$. 
By \cref{p:needed-div-T}, $\sigma$ satisfies (\ref{i:n1}), (\ref{i:n2}), (\ref{i:n3}) or (\ref{i:n4}) of \cref{p:needed-div-T}.   

We also assume that $\sigma\cup\{\be_{ij}\}$ is a face of $\LL^2_q$. Observe that if  $\be_{cc}\in \sigma\cup\{\be_{ij}\}$ for some $c$ then $\sigma\cup\{\be_{ij}\}\subseteq G_c$, by the definition of $\LL^2_q$ (see Definition \ref{d:L2}). In particular, if \cref{p:needed-div-T}(\ref{i:n1}) holds, then, since $\be_{jj}\in \sigma$, we must have $\sigma\subseteq G_j$ and hence $\be_{ia}\in G_j$. This is not possible, since $a\notin \{i,j\}$ and $i\ne j$. Similarly, if  (\ref{i:n2}) holds, then, since $\be_{ii}\in \sigma$, we must have $\sigma\subseteq G_i$  and hence $\be_{jb}\in G_i$. This is not possible, since $b\ne i$ and $j\ne i$. Hence $\sigma$ cannot satisfy (\ref{i:n1}) or (\ref{i:n2}), so it must satisfy (\ref{i:n3}) or (\ref{i:n4}) of \cref{p:needed-div-T}. Thus $i=1$.

Assume \cref{p:needed-div-T}(\ref{i:n3a}) holds. Since $i=1=j$, we have must have $\sigma\subseteq  G_1$. This happens only when $t_k=1$ for all $k$. In this case, \ref{i} holds with $j=1$. 

Assume \cref{p:needed-div-T}(\ref{i:n3b}) holds. Then $j>s$ and $j \in \{t_2,\ldots,t_s\}$ where $t_k \in \{1,j,k\}$ for $2 \leq k \leq s$. If $t_k=k$ for some $k\in \{2,\dots, s\}$, then $\be_{kk}\in \sigma\cup\{\be_{1j}\}$, hence $\be_{1j}\in G_k$. This implies $k=j$, a contradiction since $j>s$. Hence we must have $t_k\in \{1,j\}$ for all $k$. If $t_k=j$ for all $k$ then \ref{i} holds. If $t_k=1$ for some $k$, then \ref{ii} holds. 

Assume \cref{p:needed-div-T}(\ref{i:n4a}) holds. Then $j,u>s$, $j\ne u$, and $t_k \in \{1,k\}$ for $2 \leq k \leq s$. If $t_k=k$ for some $2\le k\le s$, then as above we must have $j=k$, a contradiction to $j>s$. Thus we must have thus $t_k=1$ for all $k$. Hence \ref{iii} holds. 

Assume \cref{p:needed-div-T}(\ref{i:n4b})  holds. Then $j=u >1$ and $t_k \ne 1$ for all $2 \le k\le s$. In this case $\be_{jj}\in \sigma$ and hence $\sigma\subseteq G_j$. If $j>s$, then this implies $t_2=\dots=t_s=j$, implying that (\ref{i:n3b}) holds, and as above, it follows that either \ref{i} or \ref{ii} holds. Now assume $j\le s$. Then $t_k=j$ for all $k\ne j$ with $2 \leq k \leq s$. Since $\be_{jj} \in \sigma$, then \ref{i} holds. 

 For the converse, observe that the inclusions in (i)--(iii) are special cases of the inclusions in (1)--(4) of \cref{p:needed-div-T}. Thus the converse follows by applying \cref{p:needed-div-T}. 
\end{proof}

We now revisit \cref{p:lcm-general-new}, adding a complete description of the minimal divisibility relations of $\EDsq$ when $\D$ contains at most one divisibility relation. In this case, one can use extensions to recover the entire set  $\R(\EDsq)$ of all divisibility relations on $\EDsq$ when $|\D| \leq 1$, extending the work done in~\cite{Divisibilities}. The relations (1)--(4) in \cref{p:lcm-general} below are restatements of those in \cref{p:lcm-general-new}, and can also be viewed as translating the conditions (1)--(4) in \cref{p:needed-div-T} into the language of divisibility relations. 

 \begin{theorem}[{\bf The  divisibility relations of  $\EDsq$}]\label{p:lcm-general}
 Let $I$ be a square-free monomial ideal minimally generated by
 $m_1,\ldots,m_q$. Then $I^2$ has the following divisibility relations.
  \begin{enumerate}
   \item\label{i:L2-0}   $m_im_j \mid \lcm (m_j^2,m_im_a)$  where $i,j,a\in [q]$ are such that $i<j$ and $a\notin\{i,j\}$; 
  \item\label{i:L2-1}   $m_im_j \mid \lcm (m_i^2,m_jm_b)$ where $i,j,b\in [q]$ are such that $i<j$ and $b\ne i$. 
\end{enumerate}
 If  $I$  satisfies the divisibility relation 
 \begin{equation*} 
 m_1 \mid \lcm(m_2,\ldots,m_s), 
 \end{equation*} then  $I^2$ also satisfies the following divisibility relations for indices $j, u, t_2, \dots, t_s\in [q]$: 
 \begin{enumerate}[resume]
 \item \label{i:lcm1} $m_1m_j\mid \lcm(m_2m_{t_2},\ldots, m_sm_{t_s})$ where $t_k\in \{1,j,k\}$ for all $2\le k\le s$ and 
\begin{enumerate}
  \item \label{i:3a}$j=1$; or 
    \item \label{i:3b}$j>s$ and $j\in \{t_2, \dots, t_s\}$; 
\end{enumerate}
 \item \label{i:lcm2}  $m_1m_j\mid \lcm(m_2m_{t_2},\ldots, m_sm_{t_s}, m_um_j)$ where
\begin{enumerate}
    \item \label{i:4a} $j>s$, $u>s$, $j\ne u$
     and $t_k\in \{1,k\}$  for all $2\le k\le s$; or
\item\label{i:4b} $j=u>1$ and $t_k> 1$ for all $2\le k\le s$. 
\end{enumerate}
 \end{enumerate}
Moreover, if $\D$ is a set of divisibility relations and $m_i=\ed{i}$ for $1\le i\le q$, then: 
 \begin{enumerate}[label=(\roman*)]
 \item 
 \label{i} When $\D=\emptyset$ (hence $\EDsq=\Eqsq$), the set of minimal divisibility relations on $\EDsq$ consists of the relations in (\ref{i:L2-0}) and (\ref{i:L2-1}).
 \item 
 \label{ii} When $\D=\{(1,\{2,\dots,s\})\}$,  the set of minimal divisibility relations on $\EDsq$ consists of the relations in (\ref{i:L2-0}), (\ref{i:L2-1}), (\ref{i:lcm1}), (\ref{i:4a}), which are minimal,  together with those relations in (\ref{i:4b}) that are not extensions of a relation of the form (\ref{i:4b}) or (\ref{i:3b}). 
 \end{enumerate}
 \end{theorem}

 \begin{proof}
The existence of the relations \eqref{i:L2-0}  and \eqref{i:L2-1} follows directly from \cref{p:lcm-general-new}, in particular \eqref{e:L2}. The existence of the relations \eqref{i:lcm1} and \eqref{i:lcm2} also follow from \cref{p:lcm-general-new} using the arguments from the proof of \cref{p:needed-div-T} as noted in \cref{r:relns exist}. 

To show the minimality statements in \ref{i} and \ref{ii}, let $\D$ be a set of divisibility relations and assume that $m_i=\ed{i}$ for $1\le i\le q$, that is, $I=\ED$.  

First, note that by \cref{p:needed-div-empty} when $\D=\emptyset$ and by \cref{p:needed-div-T} when $\D=\{(1,\{2,\dots,s\})\}$, all minimal relations are contained in the list of relations \eqref{i:L2-0}-\eqref{i:lcm2}. Thus it suffices to show that the relations satisfy the stated minimality properties.
 
 The relations (\ref{i:L2-0}) and (\ref{i:L2-1}) are minimal by \cref{remark:generators of extremals}.  If $\D=\{(1,\{2,\dots,s\})\}$, then we are in the setting of \eqref{eq:Ae}, and hence \cref{p:needed-div-T} applies. 
 In order to show a divisibility relation from (\ref{i:lcm1}) or (\ref{i:lcm2}) is minimal, we rewrite the relation 
 as 
 \begin{equation}
 \label{e:rel}
\ed{1}\ed{j}\mid\lcm(\tau)\,,  \qwhere \tau=\begin{cases} \{\be_{2t_2}, \dots, \be_{st_s}\} &\text{if (\ref{i:lcm1}) holds}\\
\{\be_{2t_2}, \dots, \be_{st_s}, \be_{uj}\} &\text{if (\ref{i:lcm2}) holds } 
\end{cases}
\end{equation}
If the relation \eqref{e:rel} is not minimal, there exists  $\sigma \subsetneq \tau$ with $\ed{1}\ed{j} \mid \lcm(\sigma)$. Then $|\sigma| < |\tau|$ and $\sigma$ must satisfy one of the conditions (1)--(4) of  \cref{p:needed-div-T}. 

 Assume \eqref{e:rel} comes from (\ref{i:3a}). In this case, $|\tau|= s-1$ and we must have $j=1$. Hence $\sigma$ satisfies \cref{p:needed-div-T}(\ref{i:n3a}) since this is the only option that allows $j=1$. Then $|\sigma|\ge s-1=|\tau|$, a contradiction.

Assume \eqref{e:rel} comes from 
(\ref{i:3b}). In this case, $|\tau|= s-1$ and $j>s$. Observe that $\be_{jj}\notin \tau$ and $\be_{11}\notin \tau$, hence $\sigma$ cannot satisfy (\ref{i:n1}),  (\ref{i:n2}) or (\ref{i:n4b}) in \cref{p:needed-div-T}. Since any $\sigma$ satisfying (\ref{i:n4a}) in \cref{p:needed-div-T} must have an element of the form $\be_{uj}$ with $u>s$ and $j>s$ and $\tau$ does not have such an element, we see that $\sigma$ does not satisfy (\ref{i:n4a}) either. Since $j>s$ the only remaining option is that $\sigma$ satisfies \cref{p:needed-div-T}(\ref{i:n3b}). Then $|\sigma|\ge s-1 = |\tau|$, a contradiction. 

Assume \eqref{e:rel} comes from (\ref{i:4a}). In this case $|\tau|= s$ and $\be_{11}\notin \tau$ and, since $j>s$ and $j\ne u$, we also have $\be_{jj}\notin \tau$. As above, this implies that $\sigma$ cannot satisfy (\ref{i:n1}),  (\ref{i:n2}) or (\ref{i:n4b}) in \cref{p:needed-div-T}. It cannot satisfy \cref{p:needed-div-T}(\ref{i:n3a}) either, because $j>s$. If $\sigma$ satisfies  \cref{p:needed-div-T}(\ref{i:n3b}), then it contains an element of the form $\be_{kj}$ where $2\le k\le s$ and $j>s$, and no such element exists in $\tau$. Hence the only remaining option is that $\sigma$ satisfies \cref{p:needed-div-T}(\ref{i:n4a}). In this case, $|\sigma|\ge s=|\tau|$, a contradiction. 

Assume \eqref{e:rel} comes from (\ref{i:4b}) and is not an extension of a relation of the form (\ref{i:4b}) or (\ref{i:3b}). Since $\be_{11}\notin \tau$, the set $\sigma$ does not satisfy (\ref{i:n2}) in \cref{p:needed-div-T}. Since $t_k$, $j$ and $u$ are not equal to $1$, $\tau$ contains no element of the form $\be_{1a}$, and thus $\sigma$ does not satisfy \cref{p:needed-div-T}(\ref{i:n1}). Since $j>1$, $\sigma$ cannot satisfy \cref{p:needed-div-T}(\ref{i:n3a}). The only remaining option is that  $\sigma$ has the form in \cref{p:needed-div-T}(\ref{i:4a}). Then $\sigma$ contains $\be_{uj}$ with $u>s$ and $j>s$ and $u\ne j$. This is a contradiction, since no such element exists in $\tau$. 
\end{proof}

As stated in \cref{p:lcm-general}(ii), the listed relations are minimal, with the possible exception of some of the relations from~\eqref{i:4b}. Lack of minimality of a relation of type~\eqref{i:4b} can occur when choices for $t_k$ result in a relation where $m_1m_j$ divides the lcm of a set with fewer than $s$ distinct terms. Such a relation can be properly contained in a relation that avoids repeated terms. Failure of minimality can also occur when a relation satisfying (\ref{i:4b}) properly contains a relation satisfying (\ref{i:3b}).  We illustrate this in the example below.

\begin{example}
Assume $m_1\mid \lcm(m_2, \dots, m_s)$. Setting $q=5=s$ and $j=u=2$, $t_2=3$, $t_3=4$, $t_4=5$, $t_5=4$ in (\ref{i:4b}) results in the relation $m_1m_2\mid \lcm(m_2m_3, m_3m_4, m_4m_5, m_2^2)$.  However, this relation is not minimal, because it is an extension of the relation $m_1m_2\mid \lcm(m_2m_3, m_4m_5, m_2^2)$, which is obtained by setting  $j=u=2$, $t_2=3$, $t_3=2$, $t_4=5$, $t_5=4$ in (\ref{i:4b}). 

To illustrate the second type of failure of minimality, let $q=5$ and $s=4$. Setting $t_2=5=j=u$, $t_3=3$, $t_4=4$ in (\ref{i:4b}) results in the relation $m_1m_5\mid \lcm(m_2m_5, m_3^2, m_4^2, m_5^2)$. This relation is not minimal because it is an extension of $m_1m_5\mid \lcm(m_2m_5, m_3^2, m_4^2)$, which comes from  setting  $t_2=5=j$, $t_3=3$ and $t_4=4$ in (\ref{i:3b}). 

Setting $\D=\{(1,\{2,\dots,s\})\}$ and $m_i=\ed{i}$ in the above examples shows that even in the case of $\ED$ the relations of type  (\ref{i:4b}) need not be minimal. 
\end{example}

\section{ \bf Basics of Morse Theory} \label{s:Morse_basics}

Recall that our goal is to find a cell complex that supports a resolution of $\EDsq$.  We know that $\LL^2_q$ supports a resolution of $\EDsq$, which is not minimal when $\D\neq \emptyset$. When $\D \ne \emptyset$, we will use techniques from Discrete Morse Theory in \cref{s:morse matchings} to produce a smaller complex that supports a resolution of $\EDsq$. We now introduce background and definitions for Discrete Morse Theory. 

Let $V$ be a finite set and $Y \subseteq 2^V$ be a  set of subsets of $V$. We call the elements of $Y$ {\it cells}.
A cell of cardinality $n$ is called an $(n-1)$-cell. 
We define the directed graph $G_Y$ as a graph with vertex set $\{\sigma \mid \sigma \in Y\}$ and
with directed edges $E_Y$ consisting of 
$$
\sigma'\to \sigma \qwhere
\sigma,\ \sigma' \in Y, \quad 
\sigma \subseteq \sigma' \qand 
|\sigma'|=|\sigma|+1.
$$

A {\it matching} of $G_Y$ is a set $A\subseteq E_Y$ of edges of $G_Y$
with the property that each cell of $Y$ occurs in at most one edge of
$A$.  Given a matching $A$, let the directed graph $G_Y^A$ be the same graph as $G_Y$ except that the direction of the arrows in $A$ are reversed. In the edge set of $G_Y^A$, one thinks of the oriented edges in $E_Y \smallsetminus A$ as pointing down and the oriented edges $\sigma' \to \sigma$ in $A$ reversed so that in $G_Y^A$ we have $\sigma \to \sigma'$ in $A$ viewed as pointing up.
The matching $A$ is said to be {\it acyclic} if $G^A_Y$ contains no directed cycles. The cells of $Y$ that do not appear in the edges of the matching are called $A${\it -critical cells}. We recall below some standard results. 

\begin{lemma}[{\bf Cluster Lemma} {\cite[Lemma 4.2]{Jo}}]\label{clusterlemma} 
Let $Y\subseteq 2^V$, $Q$ be a poset and $\{Y_q\}_{q\in Q}$ be a partition
of $Y$ such that
$$
\text{if $\tau\in Y_q$ and $\tau'\in Y_{q'}$ satisfy $\tau' \subseteq \tau$, then $q'\le q$.}
$$
Let $A_q$ be an acyclic matching on $G_{Y_q}$ for each $q$. Then $A=\bigcup_{q\in Q}A_q$ is an acyclic matching on $G_{Y}$. 
\end{lemma}

\begin{lemma}{\cite[Lemma 3.3]{Morse}} \label{matching-lemma}
Let $Y\subseteq 2^V$.  If $v\in V$, then 
\begin{equation*}
A_Y^{v} = \big\{\sigma'\to \sigma\in E_Y\st v\in \sigma'   {\text{ and }} \sigma=\sigma'\smallsetminus \{v\}\big\}  
\end{equation*}
is an acyclic matching on $G_Y$. 
\end{lemma}

 In what follows, we will consider a subset $Y$ of $2^V$ and $N\subseteq Y$. Our goal is to construct a matching on $G_Y$ that pairs, as much as possible, the cells of $Y$ that contain a cell of $N$, namely the cells in the set $Y_N$ defined below. To do so, we partition $Y_N$ by grouping together the cells of $Y_N$ by the largest (in a given order) cell of $N$ they contain, and then we match within each of the partition components. \cref{t:Morse} spells out this general strategy for constructing matchings, together with a description of the critical cells.

\begin{setup}
\label{NY}
Let $N\subseteq Y\subseteq 2^V$ and 
\[
Y_N= \{\gamma\in Y\st \sigma\subseteq \gamma \text{ for some $\sigma\in N$}\}\,. 
\]
Let $\preccurlyeq$ denote a total order on $N$. For each $\sigma\in N$, define
\[
Y_\sigma=\left\{\gamma\in Y_N\colon \sigma=\max_{\preccurlyeq}\{\tau\in N\colon \tau\subseteq \gamma\}\right\}\,.
\]
Observe that the sets $Y_\sigma$ form a partition of $Y_N$ indexed by the poset $N$, with the property that if $\tau\in Y_\sigma$, $\tau'\in Y_{\sigma'}$ for $\sigma$, $\sigma'$ in $N$, and $\tau\subseteq \tau'$, then $\sigma\preccurlyeq \sigma'$.  
\end{setup}

 \begin{example}
We illustrate this setup with a simple example. We let $V=\{1,2,3\}$, $Y=2^V$ and $N=\{\{1\}, \{1,3\}, \{2,3\}\}$, where $\{1\}\preccurlyeq \{1,3\}\preccurlyeq \{2,3\}$. Then we have: 
\begin{gather*}
Y_N=N\cup \{\{1,2\}, \{1,2,3\}\} \qand \\
Y_{\{2,3\}}=\{\{1,2,3\}, \{2,3\}\}, \quad Y_{\{1,3\}}=\{\{1,3\}\},\quad  Y_{\{1\}}=\{\{1\}, \{1,2\}\}.
\end{gather*}
While we can easily match the cells of $Y_{\{2,3\}}$ and those of $Y_{\{1\}}$, notice that the single cell of  $Y_{\{1,3\}}$ will remain unmatched. For larger examples, matching within the sets $Y_\sigma$ will be done using \cref{matching-lemma} for some fixed choices of the elements $v\in V$, as described below. 
\end{example}

\begin{theorem}
\label{t:Morse}
Let $\omega\colon N\to V$ be a function with the property that $\omega(\sigma)\notin \sigma$ for all $\sigma\in N$. Assume the setting in \cref{NY}. Then for each $\sigma\in N$
\[
A_{Y_\sigma}^{\omega(\sigma)}=\{\tau'\to \tau\in E_{Y_\sigma}\colon \omega(\sigma)\in \tau'\text{ and }\tau=\tau'\smallsetminus\{\omega(\sigma)\}\}
\]
is an acyclic matching on $Y_\sigma$ and the set of critical cells consists of all the cells $\tau\in Y_\sigma$ with the property that $\tau\cup\{\omega(\sigma)\}\notin Y_\sigma$. Consequently,
\begin{equation}
\label{e:Matching}
\M\coloneq \bigcup_{\sigma\in N}A_{Y_\sigma}^{\omega(\sigma)}
\end{equation}
is an acyclic matching on $G_Y$, with the following set of critical cells: 
\[
(Y\smallsetminus Y_N)\cup \bigcup_{\sigma\in N}\left\{\tau\in Y_\sigma\colon \tau\cup\{\omega(\sigma)\}\notin Y_\sigma\right\}\,.
\]
\end{theorem}

\begin{proof}
Use \cref{matching-lemma} to see that $A_{Y_\sigma}^{\omega(\sigma)}$ is an acyclic matching and then \cref{clusterlemma} to see that $\M$ is an acyclic matching.

We now prove the statement about the critical cells for the matching $A_{Y_\sigma}^{\omega(\sigma)}$.  
Let $\sigma\in N$ and $\tau\in Y_\sigma$. The hypothesis that  $\omega(\sigma)\notin \sigma$ implies 
\[
 \tau\smallsetminus \{\omega(\sigma)\}\in Y_\sigma \,.
\]
Thus, if $\omega(\sigma)\in \tau$, then \[
\tau\to \tau\smallsetminus \{\omega(\sigma)\}\in A_{Y_\sigma}^{\omega(\sigma)}
\]
and hence  $\tau$ is not a critical cell. 

If $\omega(\sigma)\notin \tau$, set $\tau'\coloneq\tau\cup \{\omega(\sigma)\}$. Note that 
 \[
 \tau'\to \tau\in A_{Y_\sigma}^{\omega(\sigma)} \iff \tau'\in Y_\sigma
 \]
 and hence $\tau$ is a critical cell if and only if $\tau'\notin Y_\sigma$. 
 
 Finally, the set $Y\smallsetminus Y_N$ needs to be included in the set of $\M$-critical cells, since the definition of $\M$ does not involve any of the vertices in this set. 
\end{proof}

We will be using a special type of acyclic matching,  namely a {\it homogeneous} acyclic matching, on the set $Y$ of cells of a simplicial or cellular complex whose vertices and cells are labeled with monomials and their lcms, respectively. In this context, the matching is  homogeneous when $\lcm(\sigma')=\lcm(\sigma)$ whenever $\sigma'\to\sigma$ is an oriented edge of the matching.  When starting with a simplicial complex, we will refer to the cells $\sigma$ as faces. This will allow us to distinguish more easily between the faces of the original simplical complex and the cells of the cellular complex that results from a matching.

The following special case of results of Batzies and Welker \cite[Proposition~1.2, Theorem~1.3, Corollary 7.6]{BW} translates information given by a homogeneous acyclic matching in a simplicial complex supporting the resolution of an ideal into a smaller cell complex that supports a free resolution of the ideal.

\begin{theorem}[{\bf Resolutions from homogeneous acyclic matchings}, see \cite{BW}]\label{t:BW} 
Let $I$ be a monomial ideal and let $\Delta$ be a simplicial complex supporting a free resolution of $I$, whose vertices and faces are labeled with monomials and their lcms, respectively. If $A$ is a homogeneous acyclic matching on $G_\Delta$, then there is a cell complex $\mathcal{X}_A$ supporting a free resolution of $I$ whose $i$-cells are in one-to-one correspondence with the $A$-critical $i$-faces of $\Delta$.
Moreover, the free resolution supported on the Morse complex $\mathcal{X}_A$ is minimal provided that, whenever $\sigma$ is an $A$-critical face and $\sigma\subseteq \sigma'$ and $\dim \sigma=\dim\sigma'-1$, then $\lcm(\sigma)\ne \lcm(\sigma')$. 
\end{theorem}

Note that in \cref{t:BW} above, when the resolution supported on $\mathcal{X}_A$ is minimal,  the Betti numbers $\beta_i(I)$ are given by the number of $A$-critical $i$-faces of $\Delta$, since such faces are in one-to-one correspondence with the $i$-cells of $\mathcal{X}_A$. Thus, the projective dimension $\pd(I)$ is the dimension of the largest $A$-critical face of $\Delta$. 

\begin{remark}
\label{r:BW}
The differential of the cell complex $\mathcal {X}_A$ is described in \cite[Lemma~7.7]{BW}, and it involves the concept of gradient paths, which will come up later. We observe here that if the set $\Gamma$ of $A$-critical faces of $\Delta$ is a simplicial complex, then it can be seen that the differential of the cell complex $\mathcal {X}_A$ described in \cite[Lemma~7.7]{BW} coincides with the differential of the homogenized simplicial complex of $\Gamma$, up to the identifying faces of $\Gamma$ with cells of $\mathcal{X}_A$. See also \cite[Proposition~5.2]{CK24}.
\end{remark}

In the setting of \cref{t:BW}, where the elements of $Y$ are faces of a simplicial complex, we describe an ordering on the cells of the Morse complex
 $\mathcal{X}_A$ in terms of gradient paths, following \cite{BW}.  A
 gradient path of length $n$  in the graph $G^A_Y$ is a directed path
$$\mathcal P \colon \sigma_0\to \dots\to \sigma_n$$ from $\sigma_0$ to $\sigma_n$; see, e.g.,
\cite[p.~165]{BW}.  For faces $\sigma, \tau$ of $Y$, the set of all
gradient paths in $G^A_Y$ from $\sigma$ to $\tau$ is denoted by $\text{GradPath}_A(\sigma,\tau)$.

If $\sigma, \tau$ are $A$-critical
faces of the simplicial complex $\Delta$ with $\dim(\sigma)=\dim(\tau)-1$, let $\sigma_A$ and $\tau_A$ be the corresponding cells in $\mathcal{X}_A$. We write $\sigma_A \leq \tau_A$ when $\sigma_A$ is contained in the closure of $\tau_A$. 
Batzies and Welker \cite[Proposition~7.3]{BW} proved: 
\begin{equation}\label{e:cells}
\sigma_A \leq \tau_A \iff 
  \text{GradPath}_A(\tau,\sigma) \neq \emptyset.
\end{equation}
Observe that in this setting, an inclusion $\sigma_A \subset \tau_A$ corresponds to a gradient path of length one.

\section{\bf Homogeneous Morse matchings for $\D \ne \emptyset$}\label{s:morse matchings}
For an ideal generated by $q$ monomials $m_1, \ldots, m_q$, the simplicial
complex $\mathbb{L}_q^r$ provides an upper bound for simplicial complexes
that support a free resolution of $I^r$~(\cite{Lr}). When $r = 1$,
$\mathbb{L}_q^1$ is the Taylor simplex $\mathbb{T}_q^1$; when $r = 2$,
$\mathbb{L}_q^2$ is defined in \cref{d:L2}. The question we
pursue in this section is whether, with the extra information
\[
  m_1 \mid \operatorname{lcm}(m_2, \ldots, m_s),
\]
we can always replace $\mathbb{L}_q^1$ and $\mathbb{L}_q^2$ with smaller
complexes.

To address this, we describe a structure supporting a resolution of $I^2$
when $I$ is generated by square-free monomials $m_1, \ldots, m_q$ such that
$m_1 \mid \operatorname{lcm}(m_2, \ldots, m_s)$. For $s = 3$, these
structures are illustrated by the smaller sucomplexes of $\LL^2_q$ in  \cref{f:L2-box} when $q = 3$ and in
\cref{f:two-triangles-edge} when $q = 4$. The results of this section will be used to
explain, in particular, why these subcomplexes support a resolution of $I^2$ when
$m_1 \mid \operatorname{lcm}(m_2, m_3)$. Our main tool for this purpose
will be discrete Morse theory, which will allow us to delete the extra
faces of $\mathbb{L}_q^2$ in an organized fashion.

\begin{theorem}[{\bf Pruning $\LL^1_q = \TT^1_q$ using $m_1\mid \lcm(m_2,\dots, m_s)$}]\label{t:L1q-prune} 
Suppose $q\ge s\ge 3$ and $I$ is minimally generated by the square-free monomials $m_1,m_2,\ldots,m_q$ such that $$m_1 \mid \lcm(m_2,\dots, m_s).$$ 
Let $\Gamma$ be the simplicial complex consisting of all faces of $\TT^1_q$ that do not contain the face $\sigma=\{\be_2, \dots, \be_s\}$. Then $\Gamma$ supports a free resolution of $I$. 

Furthermore, this resolution is minimal when $I=\ED$ where 
$$\D = \big \{(1, \{2, \ldots, s\}) \big \} \cup \big \{(b, \{2, \ldots, s\})\colon b\in J \big \} \qforsome J\subseteq \{s+1,\dots, q\}$$
where $J$ could be the empty set.
\end{theorem}
  \begin{proof}

We use \cref{t:Morse}, with $V=\N_q^1$, $Y=\TT^1_q$, and $N=\{\sigma\}$, and $\M$ as defined in \eqref{e:Matching}. Since $N$ is a one-element set, we use the trivial total order.
Then $Y_N=\{\gamma\in \TT^1_q\colon \sigma\subseteq \gamma\}=Y_\sigma$. We define $\omega(\sigma)=\be_1$.  Observe that if $\gamma\in Y_\sigma$, then $\gamma\cup\{\be_1\}\in Y_\sigma$, and hence 
$$\M = \{ \tau' \rightarrow \tau \st \be_1 \in \tau' \in Y_{\sigma} \qand \tau = \tau' \smallsetminus \{\be_1\}\}. $$
The set of $\mathcal M$-critical faces given by \cref{t:Morse} is $\TT^1_q\smallsetminus Y_N=\Gamma$. 

Noting that 
$m_1\mid \lcm(\sigma)$, we have that $\M$ is a homogeneous matching on the face poset of $\TT^1_q$ labeled with the generators of $I$. \cref{t:BW} gives a cell complex, whose cells are in one-to-one correspondence with the faces of $\Gamma$, that supports a free resolution of $I$. Moreover, since $\Gamma$ is a simplicial complex, it follows that $\Gamma$ supports a free resolution of $I$, see \cref{r:BW}. 

Now assume $I=\ED$ with $\D$ as in the statement.   According to \cref{t:BW}, to show that the free resolution supported on $I$ is minimal we need to show  that $\lcm(\tau)\ne \lcm(\tau\smallsetminus \{u\})$ for all $\tau\in \Gamma$ and all vertices $u\in \tau$. Indeed, assume $\tau=\{\be_{j_1},\dots, \be_{j_t}\}$ and $\e_{\D,j_1}\mid \lcm(\e_{\D,j_2}, \dots, \e_{\D,j_t})$. Then \cref{l:specialD} gives that $j_1\in J$ and $\{2,3,\dots, s\}\subseteq \{j_2, \dots, j_t\}$. Recall that $\sigma =\{\be_2,\be_3,\dots, \be_s\}$ and no face of $\Gamma$ contains $\sigma$. But $\tau=\{\be_{j_1},\dots, \be_{j_t}\}$ contains $\sigma$, a contradiction. 
\end{proof}

Two examples of ideals for which the divisibilities have the form above follow. Note that although the ideals are more general than $\ED$, computations in Macaulay2~\cite{M2} indicate that in both cases the resolution supported on $\Gamma$ is still minimal.

\begin{example}\label{running_ex2} Returning to \cref{running_ex}, set $I_1 = (ab, bcd, aef, cg)$. Notice that $m_1 \mid \lcm(m_2, m_3)$. Applying \cref{t:L1q-prune} (with $J=\emptyset$), shows that a free resolution of $I_1$ is supported on the simplicial complex illustrated in \cref{f:four}, which consists of all faces of $\TT^1_q$ that do not contain the face $\sigma=\{\be_2,\be_3\}$.
\end{example}

 \begin{example} \label{two relations} (Two divisibility relations). Let $I_2=(ab, bcd, aef, ce)$. Set $m_1 =ab; m_2=bcd; m_3=aef; m_4 = ce$.  Then $m_1$ and $m_4$ divide $\lcm(m_2,m_3)$ as in Theorem \ref{t:L1q-prune}. Specifically, we have $q=4$, $\D = \{(1,\{2,3\}), (4,\{2,3\})\}$, and $J = \{4\}$. 
 Applying \cref{t:L1q-prune} shows that a free resolution of $I_2$ is also supported on the simplicial complex  which consists of all faces of $\TT^1_q$ that do not contain the face $\sigma=\{\be_2,\be_3\}$. Thus both $I_2$ and the ideal $I_1$ of \cref{running_ex2} have resolutions supported on the same complex (using different monomial labels), and it is straightforward to verify that the resolutions are minimal.

 Note that using Macaulay2 \cite{M2}, one can easily verify that ${I_1}^2$ and ${I_2}^2$ have different Betti numbers, and thus one complex cannot support minimal resolutions of both. We will revisit this in \cref{s:geometric}. 
\end{example}

We have now reached the point where we are able to apply the Morse matching techniques from \cref{s:Morse_basics} to the square of an ideal. Since this proof is rather technical, we follow it up with \cref{ex:MorseL2}, which describes the Morse matching explicitly. The reader may find it useful to consult this example while reading through the proof.

  \begin{theorem}[{\bf Pruning $\LL^2_q$ using one divisibility relation}]\label{t:L2q-prune} 
Suppose $q\ge s\ge 3$ and $\delta=(1, \{2,\ldots, s\})$ is a divisibility relation. Then there is an acyclic matching $\M_{q,\delta}$ on $\LL^2_q$ such that the set of $\M_{q,\delta}$-critical faces consists of all those 
$\tau\in\LL^2_q$ satisfying either one of the conditions below: 
\begin{enumerate}
\item[(a)] 
\label{type a}$\tau$ does not contain any face of the form 
\[
\{\be_{2t_2}, \dots, \be_{st_s}\} \qwith \{t_2, \dots, t_s\}=\begin{cases}
\{1,j\} &\text{with $s<j\le q$} \qquad \text{or} \\
\{j\} &\text{with $j\in [q]$;} 
\end{cases}
\]
\item[(b)] 
\label{type b}$\tau=\{\be_{21},\be_{31},\dots, \be_{s1}\}\cup\gamma\cup \gamma'$ with 
$$\emptyset\ne\gamma\subseteq 
\{\be_{ik}\colon  1<i<k\le s\}\qand \gamma'\subseteq \{\be_{1(s+1)}, \dots, \be_{1q}\}.$$
\end{enumerate}
\end{theorem}

\begin{proof}
We will use our general Morse theorem, \cref{t:Morse}, with $V=\N_q^2$, $Y=\LL^2_q$, and $N$ equal to the set of faces $\sigma$ of the following three distinct types: 
\begin{enumerate}[label=(\roman*)]
\item \label{i}$\sigma=\{\be_{2j},\be_{3j}, \dots, \be_{sj}\}$ for some $j\in [q]$; 
\item  \label{ii} $\sigma=\{\be_{2t_2},\dots, \be_{st_s}\}$ for some $t_2, \dots, t_s$ with $\{t_2, \dots, t_s\}=
\{1,j\}$, and $j>s$;
\item \label{iii} $\sigma=\{\be_{21},\be_{31},\dots, \be_{s1}, \be_{ju}\}$ for some $u$ with $j> u>s$
\end{enumerate}
and $Y_N= \{\gamma\in Y\colon \sigma\subseteq \gamma \text{ for some $\sigma\in N$}\}$.
We define a total order on $N$ in such a way that 
\[
\text{ type (i)}\prec   \text{ type (ii)}\prec  \text{ type (iii)}\,.
\]
It is not relevant for the proof how faces of the same type compare to each other, so we do not need to be more precise in defining this order. Observe that the integer $j$ in (i)\,--\,(iii) is uniquely determined by $\sigma$, so we will write $j_{\sigma}$ when we want to emphasize this dependence. 

We now define $\omega\colon N\to V$
by taking $\omega(\sigma)=\be_{1j}$ in all three cases, where $j=j_\sigma$. Note that in each of the three cases, $\be_{1j} \not\in \sigma$ and $\sigma \cup \{\be_{1j}\} \in \LL^2_q$. 
To see the latter, first recall that the facets of $\LL^2_q$ are of the form $\mathcal{B}$ or $G_a$ with $a \in [q]$ (see \cref{d:L2} for the notation).
In case~(i) if  $j\le s$, then $\sigma \subset G_j$ and $\be_{1j} \in G_j$, so $\sigma \cup \{\be_{1j}\} \subseteq G_j$ as well. If $j>s$, then $\sigma \subseteq \mathcal{B}$ where $\mathcal{B}$ is the set of all vertices that are square-free, meaning of the form $\be_{ab}$ with $a\ne b$. Since $j>s$ we have $\be_{1j} \in \mathcal{B}$ as well, so $\sigma \cup \{\be_{1j}\} \subseteq \mathcal{B}$ as desired. In cases (ii) and (iii), we again have $j>s$ which, together with the conditions on $t_i$ in case (ii) and the conditions on $u$ in case (iii), implies $\sigma \subseteq \mathcal{B}$ and thus the proof follows as in case (i). 

 We now describe the critical faces for the matching in \cref{t:Morse}, with the above input. According to the theorem, there are two types of critical faces: 
\begin{itemize}
\item Faces in $\LL^2_q\smallsetminus Y_N$; namely, faces in $\LL^2_q$ that do not contain any faces of type (i), (ii) or (iii) above. 
\item Faces $\tau$ in $Y_N$ satisfying $\tau\cup\{\be_{1j}\}\notin Y_\sigma$, where  $\sigma\in N$ is such that $\tau\in Y_\sigma$, and $j=j_\sigma$. 
\end{itemize}
The faces described in the first bullet are exactly the ones described in (a) in the statement, noting that any face of type (iii) contains a face of type (i). We need to show that the faces described in the second bullet above are precisely the ones described in part (b). There are two cases to consider for $\tau\in Y_\sigma$ satisfying $\tau\cup\{\be_{1j}\}\notin Y_\sigma$: 
\begin{enumerate}
\item $\tau\cup\{\be_{1j}\}\notin \LL^2_q$;
\item $\tau\cup\{\be_{1j}\}\in \LL^2_q\smallsetminus Y_\sigma$. 
\end{enumerate}

\begin{claim}
For any $\sigma\in N$, there are no faces $\tau\in Y_\sigma$ satisfying (2). 
\end{claim}

To prove the claim, let $\sigma\in N$ and $\tau\in Y_\sigma$, and assume (2) holds. The fact that $\tau\cup\{\be_{1j}\}\in \LL^2_q\smallsetminus Y_\sigma$ implies that there exists $\eta\in N$ with $\sigma\prec\eta$ such that $\tau\cup\{\be_{1j}\}\in Y_{\eta}$. Observe that we must have $\be_{1j}\in \eta$. 

If $\sigma$ is of type (ii) or (iii), then $j>s$. However, note that, when $j>s$, we have $\be_{1j}\notin \eta$ for all $\eta\in N$, a contradiction.

Assume $\sigma=\{\be_{2j}, \be_{3j}, \dots, \be_{sj}\}$ is of type (i). If $j>s$, then use the same argument as above to reach a contradiction. If $2\leqslant j\leqslant s$, then $\be_{jj}\in \sigma$. Since $\sigma\subseteq \tau$, we conclude $\be_{jj}\in \tau$.  Since $\tau\in \LL^2_q$, we have thus 
\[
\be_{1j}\in \eta\subseteq \tau\cup\{\be_{1j}\}\subseteq \{\be_{1j},\be_{2j}, \dots, \be_{qj}\}=G_j\,.
\]
However, there are no faces $\eta\in N$ that satisfy $\be_{1j}\in \eta\subseteq G_j$ , a contradiction. It remains to treat the case $j=1$. Then we must have $\be_{11}\in \eta$. However, there are no faces in $N$ that contain $\be_{11}$, a contradiction. Thus all faces must be of the form (1). 

\begin{claim}
Let $\tau\in Y_\sigma$ for some $\sigma\in N$.  Then $\tau$ satisfies (1) if and only if
\begin{equation}
\label{e:tau}
\tau=\{\be_{12},\be_{13},\dots, \be_{1s}\}\cup\gamma\cup \gamma'\qwith\emptyset\ne \gamma\subseteq \{\be_{ik}\colon  1<i<k\le s\}\text{ and } \gamma'\subseteq \{\be_{1(s+1)}, \dots, \be_{1q}\}
\end{equation}
\end{claim}

To prove the claim, assume $\tau$ satisfies \eqref{e:tau}. Note that $\sigma\subseteq \tau$ and $\sigma\in N$. If $\sigma$ is of type (ii) or type (iii), then both $j,u>s$, so $\be_{kj}, \be_{ju} \not\in \gamma \cup \gamma'$ when $2 \leq k \leq s$. Thus the only option for $\sigma$ is $\sigma=\{\be_{12}, \dots, \be_{1s}\}$, and hence $j=1$. Then $\omega(\sigma) = \be_{11}$ and if $\tau\cup\{\be_{11}\}\in \LL^2_q$, then $\tau\cup\{\be_{11}\}\subseteq G_1$, but $\gamma \ne \emptyset$, and $\gamma \subset \tau$ but $\gamma \cap G_1 = \emptyset$, a contradiction. Thus $\tau \cup \{\be_{11} \} \not\in \LL_q^2$.

Conversely, assume that $\tau\cup\{\be_{1j}\}\notin \LL^2_q$.  Since any face contained in $\mathcal B$ is contained in $\LL^2_q$, we must have that either $\tau\not\subseteq \mathcal B$ or $\be_{1j}\not\in \mathcal B$. 

If $\be_{1j}\notin \mathcal B$, and hence $j=1$, , then  $\sigma=\{\be_{12},\be_{13},\dots, \be_{1s}\}$ is a type (i) face of $N$. Since $\sigma\subseteq \tau$ and $\tau\in \LL^2_q$, but $\tau\cup\{\be_{11}\}\notin \LL^2_q$, we must have $\tau\subseteq \mathcal B$ and $\tau$ must contain an element $\be_{ik}$ with $1<i<k\le q$. If $i>s$, then $\rho=\{\be_{12}, \dots, \be_{1s}, \be_{ik}\}$ is a type (iii) face of $N$, and it is contained in $\tau$. Since $\sigma\prec\rho$ in our order, this contradicts the fact that $\tau\in Y_\sigma$. If $i\le s$ and $k>s$, then $\rho=\{\be_{ik}\}\cup\left(\{\be_{12}, \be_{13}, \dots, \be_{1s}\}\smallsetminus \{\be_{1i}\}\right)$ is a type (ii) face of $N$ contained in $\tau$. However, $\sigma\prec \rho$, contradicting $\tau\in Y_\sigma$. Thus we must have $k\le s$ for all $\be_{ik}\in \tau$ with $1<i<k$, and there is at least one such element in $\tau$. This element is in $\gamma$, so $\gamma \ne \emptyset$. Now if for $i<k$ we have $\be_{ik}\in \tau\smallsetminus \sigma$, then as above if $i>1$ then $k\le s$ and $\be_{ik} \in \gamma$. If $i=1$, then $\be_{1k} \in \sigma$ or $\be_{1k} \in \gamma'$. We conclude that $\tau$ must have the form in \eqref{e:tau}.  

It remains to consider the case when $\tau\not\subseteq\mathcal B$. In this case, we must have  $\be_{ii}\in \tau$ for some $i$. Since $\tau\in \LL^2_q$, we must have $\tau\subseteq G_i$. Since $\sigma\subseteq \tau$ and $\sigma\in N$, we must have $\sigma=\{\be_{2i},\be_{3i}, \dots, \be_{si}\}$ and $j=i$. Then $\tau\cup \{\be_{1j}\}$ is contained in $G_i$, and hence it is in $\LL^2_q$, a contradiction. This finishes the proof of Claim 2, and thus the description of $\M_{q,\delta}$-critical faces. 
\end{proof}

\begin{remark}
\label{r:arrowsinG1}
Let $\M_{q,\delta}$ be the matching defined in \cref{t:L2q-prune} and let $N$ and $\omega\colon N\to \mathcal N_q^2$ be as defined in the proof of this result. If  $\tau\in G_1$ or $\tau'\in G_1$ , then 
\begin{equation}
\tau'\to \tau\in \M_{q,\delta}\iff \be_{11}\in \tau'\,,\,\,\tau=\tau'\ssm \{\be_{11}\}\qand \{\be_{21},\be_{31}, \dots, \be_{s1}\}\subseteq \tau\,.
\end{equation}

Indeed, as described in \cref{t:Morse}, the fact that $\tau'\to \tau$ is in the matching means that $\tau, \tau'\in Y_\sigma$ for some $\sigma\in N$ with $\omega(\sigma)\in \tau'$  and  $\tau=\tau'\ssm \{\omega(\sigma)\}$. The statement follows from the observation that if $\alpha\in G_1$ and $\sigma\in N$, then  $\alpha\in Y_\sigma$ if and only if  $\sigma=\{\be_{21},\be_{31}, \dots, \be_{s1}\}$ (in which case $\omega(\sigma)=\be_{11}$) and $\sigma\subseteq \alpha$. 
\end{remark}

\begin{example}
\label{ex:MorseL2}
Suppose in \cref{t:L2q-prune} that $q=4$ and $s=3$, which means
$$I=(m_1,m_2,m_3,m_4)
\qand 
m_1 \mid \lcm(m_2, m_3).$$ 
    We illustrate the matching in \cref{t:L2q-prune} which removes faces of $\LL^2_4$,  leading to the Morse complex $\LDsq$ in this particular case.  To simplify notation, below we will write $ij$ instead of $\be_{ij}$. 
    The set $N$ consists of elements of types (i), collected in the first set of the union below, and type (ii), collected in the second set of the union. In this case, there are no elements of type (iii). In particular, 
\begin{align*}
N = &\big\{\{2j,3j\} \st j \in [4]\big\} 
\,\cup \,
\big\{\{2t_2,3t_3\} \st  \{t_2, t_3\}= \{1, 4\} \big\} \\
=&\{\{12,13\}, \{22,23\}, \{23,33\}, \{24,34\}, \{12,34\}, \{13,24\}\}
\end{align*}
where the second display fixes the order $\preccurlyeq$ on $N$ so that $\text{ type (i)}\prec   \text{ type (ii)}$ as required by the proof of \cref{t:L2q-prune}.  With $Y = \LL^2_4$ (see \cref{f:L24}),  recall that 
\[
Y_N= \big \{\gamma\in Y\st \sigma\subseteq \gamma \text{ for some $\sigma\in N$} \big \}
\qand
Y_\sigma=\big\{\gamma\in Y_N\colon \sigma=\max_{\preccurlyeq}\{\tau\in N\colon \tau\subseteq \gamma\}\big\}\,.
\]
  Note that 
  $$
  Y_{\{12,13\}}=\big\{
  \{12,13\}, \{12,13,14\}, \{12,13,23\}, 
  \{12,13,11\}, \{12,13,14,11\}, \{12,13,14,23\}
  \big\}.
  $$ 
  More specifically, sets like $\{12, 13, jj\}$ for $j \neq 1$ are not faces in $\LL^2_4$, hence they do not belong to $Y_{\{12,13\}}$, along with any sets that contain them.  Next, although $\sigma = \{12,13\}$ is contained in $\gamma = \{12,13,24\}$, $\sigma \neq \max_{\preccurlyeq}\{\tau\in N\colon \tau\subseteq \gamma\}$. The maximum in this case is $\{13,24\}$. Likewise, for $\gamma = \{11,12,13,34\}$ the maximum is $\{12,34\}$. Therefore, no face belonging to $Y_{\{12,13\}}$ can contain the elements $24$ or $34$.  This argument leaves the subsets shown above, which can be verified to belong to $Y_{\{12,13\}}$.  The $Y_{\sigma}$ for other sets $\sigma$ are computed similarly and shown below:
\begin{align*}
Y_{\{12,13\}}=&\big\{\{12,13\}, \{12,13,14\}, \{12,13,23\}, \{11,12,13\}, \{11,12,13,14\}, \{12,13,14,23\}\big\} \\
Y_{\{22,23\}}=&\big\{\{22,23\},\{12,22,23\},\{22,23,24\}, \{12, 22, 23,24\}\big\}\\
Y_{\{23,33\}}=&\big\{\{23,33\},\{13, 23,33\},\{23,33,34\}, \{13,23,33,34\}\big\}\\
Y_{\{24,34\}}=&\big\{\{24,34\}, \{24,34,14\},\{24,34,44\}, \{24,34,23\}, \{24,34,14,44\},\{24,34,14,23\}\big\}\\
Y_{\{12,34\}}=&\big\{\{12,34\}, \{12,34,13\}, \{12,34,14\}, \{12,34,23\}, \{12,34, 24\}, \{12,34,13,14\}, \\
& \{12,34,13,23\}, \{12,34, 14,23\}, \{12,34,14,24\}, \{12,34,23,24\}, \\
&\{12,34,13,14,23\}, \{12,34,14,23,24\}\big\}\\
Y_{\{13,24\}}=&\big\{\{13,24\}, \{13,24,12\},  \{13,24,14\},  \{13,24,23\},  \{13,24,34\},
\{13,24,12,14\}, \\
& \{13,24,12,23\},  \{13,24,12,34\},  \{13,24,14,23\}, \{13,24,14,34\}, \{13,24,23,34\}, \\
& \{13,24,12,14,23\}, \{13,24,12,14,34\}, \{13, 24, 12, 23,34\}, \{13,24,14,23,34\}, \\
& \{13,24,12,14,23,34\}\big\} \\
\end{align*}
We have 
\begin{align*}
    \omega(\sigma)=\begin{cases}
    11 &\text{if $\sigma\in Y_{\{12,13\}}$}\\
    12 &\text{if $\sigma\in Y_{\{22,23\}}$}\\
    13 &\text{if $\sigma\in Y_{\{23,33\}}$}\\
    14 &\text{if $\sigma\in Y_{\{24,34\}} \cup Y_{\{12,34\}}\cup Y_{\{13,24\}}$}\\
    \end{cases}
\end{align*}
 The matchings within $Y_{\{12, 13\}}$, per \cref{t:Morse}, are
$$
\{11,12,13\}\to \{12,13\} 
\qand 
\{11,12,13,14\}\to \{12,13,14\}.
$$
Therefore, there are two critical faces remaining in $Y_{\{12,13\}}$, namely 
$\{12,13,23\}$ and $\{12,13,14,23\}$.  On the other hand, the matchings within $Y_{\{13,24\}}$ are as follows: 
\begin{center}
\begin{tabular}{ c c }
 $\{13,24,14\}\to \{13,24\}$& $\{13,24,12,14\}\to\{13,24,12\}$ \\
 $\{13,24,14,23\}\to\{13,24,23\}$ &  $\{13,24,14,34\}\to \{13,24,34\}$\\
 $\{13,24,12,14,23\}\to\{13,24,12,23\}$& $\{13,24,12,14,34\}\to \{12,24,12,34\}$\\
$\{12,24,14,23,34\}\to \{12,24,23,34\}$ & $\{13,23,12,14,23,34\}\to \{13,23,12,23,34\}$\\ 
\end{tabular}
\end{center}
As a result, there are no critical faces left in $Y_{\{13,24\}}$ after the matching. Similarly, one can see that all faces of $Y_{\{12,34\}}$, $Y_{\{23,33\}}$, $Y_{\{22,23\}}$, $Y_{\{24,34\}}$ get matched. 
Thus the critical faces are the faces in $\LL^2_4$ that do not contain any of the faces in $N$ (see Section 6), union with the faces $\{12,13,23\}$ and $\{12,13,14,23\}$. These last two faces are of type (b) in the theorem, and the others are of type (a). 
\end{example}

By \cref{t:BW}, the matching $\mathcal{M}_\delta$ in \cref{t:L2q-prune} has an associated cell complex $\mathcal{X}_{\M_{q,\delta}}$, which will be the focus of the remainder of the paper. To make clear the connection to the complex $\LL_q^2$, we introduce a notation for this complex. Note that the labels of this complex will depend on the choice of the ideal $I$. However, for any choice of $I$ satisfying the relation $\delta$, the underlying complex is the same. Thus we define the complex without labels and will indicate the choice of labels when using the complex.

\begin{definition}[{\bf The cell complex $\LqDsq$}]\label{d:LD}
Suppose $q\ge s\ge 3$ and $\delta=(1,\{2,\ldots,s\})$ is a divisibility relation. Let $\M_{q,\delta}$ be the matching defined in \cref{t:L2q-prune}. If $\D=\{\delta\}$, we define $\LDsq$  to be the cell complex $\mathcal{X}_{\M_{q,\delta}}$ whose cells are in one-to-one correspondence with the $\M_{q,\delta}$-critical faces of $\LqDsq$.
\end{definition}

\begin{theorem}[{\bf Resolutions supported on $\LqDsq$}]\label{t:minimality}
Suppose $q\ge s\ge 3$ and $I$ is minimally generated by square-free monomials $m_1,m_2,\ldots,m_q$ such that
$m_1 \mid \lcm(m_2,\dots, m_s).$
If $\D=\{(1, \{2,\dots, s\}\})$, then the cell complex $\LDsq$ labeled with the monomial generators of $I^2$ supports a free resolution of $I^2$ and hence the $f$-vector of $\LDsq$ provides an upper bound for the Betti numbers of $I^2$. 

Moreover, the resolution is minimal when $I=\ED$. 
\end{theorem}

\begin{proof}
 We label $\LL^2_q$ using the monomial generators $m_im_j$ of $I^2$.  Let $N$ and $\omega\colon N\to \N_q^2$ be as defined in the proof of \cref{t:L2q-prune}. For all $\sigma\in N$, by \cref{p:lcm-general-new} $\sigma$ has the additional property that $m_1m_j\mid \lcm(\sigma)$. Since the label of $\omega(\sigma)$ is $m_1m_j$, we see that the matching $\M_{q,\delta}$ defined in \cref{t:Morse} is homogeneous. 

The fact that $\LDsq$ (labelled with the monomial generators of $I^2$) supports a free resolution of $I^2$ follows from \cref{t:BW}. The entries of the $f$ are the ranks of the free modules in this resolution, which are necessarily upper bounds for the corresponding Betti numbers of $I^2$. 

To see the remaining statement, let $I=\ED$. Recall that the generators of $\ED$ satisfy $\D$. Consider the Morse complex $\LDsq$ from \cref{d:LD}, labeled with the monomial generators of $\EDsq$. 
By \cref{t:BW},  the free resolution supported on the Morse complex $\LDsq$ is minimal provided that, whenever $\sigma$ is an $\mathcal M_\delta$-critical face and $\sigma'=\sigma\smallsetminus\{\be_{ij}\}$ for some $\be_{ij}\in \sigma$ we have $\lcm(\sigma)\ne \lcm(\sigma')$. 

Indeed, assuming that $\lcm(\sigma)=\lcm(\sigma')$ for $\sigma, \sigma'$ as above, it follows that $\ed{i}\ed{j}\mid \lcm(\sigma')$. Assume $i\le j$. By \cref{p:needed-div}, we have that $i=1$ and one of the following holds: 
\begin{enumerate}[label=(\roman*)]
\item \label{i}$\{\be_{2j},\be_{3j}, \dots, \be_{sj}\}\subseteq \sigma'$; 
\item  \label{ii} $\{\be_{2t_2},\dots, \be_{st_s}\}\subseteq \sigma'$ for some $t_2, \dots, t_s$ with $\{t_2, \dots, t_s\}=
\{1,j\}$, and $j>s$;
\item \label{iii} $\{\be_{21},\be_{31},\dots, \be_{s1}, \be_{ju}\}\subseteq \sigma'$ for some $u$ with $j\ne u>s$, and $j>s$.   
\end{enumerate}
In particular, it is clear that $\sigma$ is  not a critical face of type (a) of  \cref{t:L2q-prune}, so it must be of type (b), so $\sigma=\{\be_{21},\be_{31},\dots, \be_{s1}\}\cup\gamma\cup \gamma'$ with 
$$\emptyset\ne\gamma\subseteq 
\{\be_{ik}\colon  1<i<k\le s\}\qand \gamma'\subseteq \{\be_{1(s+1)}, \dots, \be_{1q}\}.$$
Consequently, $\sigma$ does not contain any element of the form $\be_{\ell k}$ with $\ell>1$ and $k>s$. We see that (i) must hold, with $j\le s$. If $j>1$, we see that $\be_{jj}\in \sigma'$ and hence $\be_{jj}\in \sigma$, a contradiction. We must have thus $j=1$. In this case, $\be_{11}=\be_{ij}\in \sigma$, a contradiction.  
\end{proof}

As an application of \cref{t:minimality}, we can find the projective dimension of $\EDsq$ using the dimension of the Morse complex from \cref{t:L2q-prune}, which is the largest dimension of an $\M_{q,\delta}$-critical face of $\LL^2_q$.
Recall that when $q\ge 3$ $$\pd(\Eqsq)=\dim(\LL^2_4)=\binom{q}{2}-1.$$

\begin{theorem}[{\bf A sharp bound on $\pd(I)$ and $\pd(I^2)$ when $m_1\mid \lcm(m_2,\dots, m_s)$}]\label{t:betti-2-bound} 
Suppose $q\ge s\ge 3$ and $I$ is minimally generated by square-free monomials $m_1,m_2,\ldots,m_q$ such that
$m_1 \mid \lcm(m_2,\dots, m_s).$ 
Then
\begin{align}
\pd(I)&\leq\pd(\ED)=q-2\label{pdI}\\
\pd(I^2)&\leq\pd(\EDsq)=\begin{cases}
\binom{q}{2}-(q-s+2)&\text{if $q>s$}\\
\binom{q}{2}-1  &\text{if $q=s$} \,,
\end{cases}\label{pdI2}
\end{align}
where  $\D=\{(1, \{2,\dots, s\}\})$.
  \end{theorem}
\begin{proof} 
The first inequalities in \eqref{pdI} and \eqref{pdI2} follow from \cite[Theorem 4.4]{Divisibilities}.

To compute $\pd(\ED)$, recall from \cref{t:L1q-prune}  that $\ED$ has a minimal free resolution supported on the simplicial sub-complex of $\TT^1_q$ whose faces do not contain $\{\be_2, \dots, \be_s\}$. The largest such face has $q-1$ elements, and so is of dimension $q-2$.

We now compute $\pd(\EDsq)$ by finding the largest dimension of a critical face from \cref{t:L2q-prune}. First assume $q=s$. Then the facet $\mathcal B$ of $\LL_q^2$ given in  \cref{d:L2} is a critical face of type (b) of \cref{t:L2q-prune} with $\gamma$ maximal and $\gamma'=\emptyset.$ Since $\mathcal B$ has $\binom{q}{2}$ elements, we have 
\begin{equation}
\label{pd-bound}\pd(\EDsq)\ge \binom{q}{2}-1 \qwhen q=s.
\end{equation}
In general, consider 
\begin{equation}
\label{sigma-def}
\sigma=\mathcal B\smallsetminus \{\be_{12}, \be_{2(s+1)}, \dots, \be_{2q}\}.
\end{equation}
 We show that $\sigma$ is a critical face of type (a) from \cref{t:L2q-prune}. Assume that 
\begin{equation}
\label{sigma-a}
\sigma\supseteq \{\be_{2t_2}, \dots, \be_{st_s}\} \qwith \{t_2, \dots, t_s\}=\begin{cases}
\{1,j\} &\text{with $s<j\le q$} \quad \text{or} \\
\{j\} &\text{with $j\in [q]$ } 
\end{cases} \,\, .
\end{equation}
If $2\le j\le s$, then $t_i=j$ for $2\le i \le s$, so $\be_{jj}\in \sigma$. Now by \eqref{sigma-def}, $\be_{jj}\in \mathcal B$, contradicting the definition of $\mathcal B$ .  If $j=1$, then $\be_{12}\in \sigma$, contradicting \eqref{sigma-def}. If $j>s$, in which case $s < q$, then $\sigma$ contains $\e_{2t_2}$ where $t_2=1$ or $t_2=j$, contradicting \eqref{sigma-def}. We conclude that \eqref{sigma-a} does not hold, and thus $\sigma$ is a critical face of type (a). Observe that the dimension of $\sigma$ is equal to  $\binom{q}{2}-(q-s+1)-1$, showing that 
\begin{equation}
\label{pdle}
\pd(\EDsq)\ge \binom{q}{2}-(q-s+2)\,.
\end{equation}

To show the reverse inequality, we show that all critical faces listed in \cref{t:L2q-prune} have cardinality at most $\binom{q}{2}-(q-s+1)$ when $s < q$ and cardinality at most $\binom{q}{2}$ when $s=q$. First let $\tau$ be a critical face of type (a). Since $\tau$ does not contain $\{\be_{21}, \dots, \be_{s1}\}$, we know that $\be_{k1}\notin \tau$ for some $2 \le k\le s$. Furthermore, when $s < q$, the fact that  $\tau$ does not contain $\{\be_{2j}, \dots, \be_{sj}\}$ when $j>s$ implies that for each $j>s$ there exists $i_j\in \{2,\dots, s\}$ such that $\be_{i_jj}\notin \tau$. If $\tau \subseteq \mathcal B$, then $\tau$ can have at most $\binom{q}{2}-1$ elements if $q=s$ and at most $\binom{q}{2}-(q-s+1)$ elements if $q > s$. If $\tau \subseteq G_i$, then setting $j=i$ since $\tau \not\supseteq \{\be_{2i},\ldots,\be_{si}\}$, there is an element of $G_i$ not in $\tau$, so $\tau$ can have at most $q-1$ elements for any $q \ge s$. 

 Finally, we compare the various sizes of critical faces found above. If $q=s$ note that since $q \geq 3$ we have $\binom{q}{2}\ge q-1.$ Thus when $q=s$, the maximum cardinality of a critical face is $\binom{q}{2}$.

 Assume $q >s$.
 The maximum cardinality of a face of type (b) from \cref{t:L2q-prune} is 
 \begin{align*}
 (s-1)+\binom{s-1}{2}+ (q-s)&=\binom{s-1}{2}+q-1=\frac{(s-1)(s-2)+2q-2}{2}=\frac{s^2-3s+2q}{2}.
 \end{align*}
 On the other hand, we saw above that the maximum cardinality of a face of type (a) contained in $\mathcal B$ is: 
 \begin{align*}
 \binom{q}{2}-q+s-1=\frac{q(q-1)-2q+2s-2}{2}=\frac{q^2-3q+2s-2}{2}.
 \end{align*}

A direct computation shows
 \begin{align*}
 \frac{q^2-3q+2s-2}{2}-\frac{s^2-3s+2q}{2}&=\frac{q^2-s^2-5q+5s-2}{2}=\frac{(q-s)(q+s-5)}{2}-1\,.
 \end{align*}
Since $s \ge 3$ and $q>s$, we have
 $q+s\ge 7$, so the difference is positive and 
 $$\binom{q}{2}-q+s-1$$ 
 is the larger quantity. Finally, we compare this bound to $q-1$ when $q>s$. Consider
 $$\binom{q}{2}-q+s-1 - (q-1) = \frac{q^2 -5q +2s}{2}.$$ 
 The function $q^2 -5q +2s$ is increasing in both $q$ and $s$ when $q \ge 3$, and is positive for $q \ge 4,\ s\ge 3$, so $\binom{q}{2}-q+s-1$ is the larger quantity.

 Thus, recalling that projective dimension is one less than the maximum cardinality of a critical face, we have
$$\pd(\EDsq)=\begin{cases}
\binom{q}{2}-(q-s+2)&\text{if $q>s$}\\
\binom{q}{2}-1  &\text{if $q=s$}
\end{cases}.$$
\end{proof}

In order to understand the geometry of the cell complex $\LDsq$, we need to understand the order relation among the cells, which we describe next. In the statement below, we reference the types (a) and (b) of critical cells described in  \cref{t:L2q-prune}. 

\begin{theorem}[{\bf Cell order in the Morse complex}]\label{t:cell order}
Suppose $q\ge s\ge 3$ and $\delta = (1, \{2,\ldots, s\})$. Set $A = \mathcal M_{q,\delta}$ to be the matching on $\LL_q^2$ in \cref{t:L2q-prune}.  
If $\sigma$, $\tau$ are $A$-critical faces  with $|\sigma|=|\tau|-1$, then the following hold, where $\gamma=\tau\cap \{\be_{ik}\colon 1<i<k\le s\}$: 
\begin{enumerate}
\item If $\tau$ is of type (a), then $\sigma_A\le\tau_A$ if and only if  $\sigma\subseteq \tau$.
\item If $\tau$ is of type (b) and $|\gamma|>1$, then $\sigma_A\le\tau_A$ if and only if  $\sigma\subseteq \tau$. 
\item If $\tau$ is of type (b) and $|\gamma|=1$, then $\sigma_A\le\tau_A$ if and only if  $\sigma\subseteq \tau$ or
\[
 \sigma=((\tau\smallsetminus \gamma)\cup \{\be_{11}\})\smallsetminus \{\be_{1 \ell}\} \quad {\text{ with }} 2\le \ell\le s.
\] 
\end{enumerate}
\end{theorem}

\begin{proof}
Recall from \eqref{e:cells} that $\sigma_A\le\tau_A$ is equivalent to the existence of a gradient path from $\tau$ to $\sigma$. As described for example in \cite[Lemma 3.2]{Saraetall},  the gradient path can be visualized as
\begin{equation}
\label{e:gradient}
\begin{array}{ccccccccccccc}
   \tau = \sigma_0 &&&& \sigma_2 &&&&&& \sigma_{u} &&  \\
  &\searrow  &&\nearrow &&\searrow &&&&\nearrow &&\searrow & \\ 
    &&\sigma_1&&&& \sigma_3 &\ldots&\sigma_{u-1}&&&& \sigma_{u+1}=\sigma  
    \end{array}
    \end{equation}
where arrows pointing down correspond to inclusions, with $|\sigma_i|=|\sigma_{i-1}|-1$ when $i$ is odd, and arrows pointing up correspond to edges that were reversed due to being in the matching, that is, $\sigma_2\to \sigma_1, \dots,  \sigma_{u}\to \sigma_{u-1}$ are in $A$. In each of the three cases above, if $\sigma\subseteq \tau$, then there is a gradient path of length one and hence $\sigma_A\le\tau_A$. This proves the reverse implications in (1) and (2) and part of the one in (3).

We now prove the direct implications in (1)--(3). Assume $\sigma_A\subseteq \tau_A$ and consider a gradient path as in \eqref{e:gradient}. 

(1) Assume $\tau$ is of type (a), namely $\tau$ does not contain any face of the form 
\[
\{\be_{2t_2}, \dots, \be_{st_s}\} \qwith \{t_2, \dots, t_s\}=\begin{cases}
\{1,j\} &\text{with $s<j\le q$} \qquad \text{or} \\
\{j\} &\text{with $j\in [q]$.} 
\end{cases}
\]
Since $\sigma_1\subseteq \tau$, it follows that $\sigma_1$ is also of type (a) and in particular, $\sigma_1$ is a critical face. Since, by definition, critical faces are not in the matching, there are no arrows pointing up from a critical face and we must have $\sigma_1=\sigma$ and $\sigma \subseteq \tau$. 

(2) and (3): Assume $\tau$ is of type (b), namely $\tau=\{\be_{21},\be_{31},\dots, \be_{s1}\}\cup\gamma\cup \gamma'$ with 
$$\emptyset\ne\gamma\subseteq 
\{\be_{ik}\colon  1<i<k\le s\}\qand \gamma'\subseteq \{\be_{1(s+1)}, \dots, \be_{1q}\}.$$
If $\sigma_1$ is a critical face, then, as above, we have $\sigma_1=\sigma\subseteq \tau$. Assume that $\sigma_1$ is not a critical face. In particular, $\sigma_1$ is not a critical face of type (b) and hence one of the following must hold: 
\begin{enumerate}
\item[(i)] $\sigma_1=\tau\smallsetminus\{\be_{i 1}\}$ with $2\le i\le s$; i.e., $\sigma_1 = \{\be_{21},\dots, \widehat{\be_{i 1}}, \dots, \be_{s1}\}\cup \gamma\cup \gamma'$

\item[(ii)] $\sigma_1=\tau\smallsetminus \{\be_{ik}\}$, where $\gamma=\{\be_{ik}\}$; 
i.e., $\sigma_1 = \{\be_{21},\be_{31},\dots, \be_{s1}\}\cup \gamma'.$
\end{enumerate}
It is also true that $\sigma_1$ is not a critical face of type (a), and hence there exists $j$ such that 
\begin{equation}
\label{e:tau1}
\tau\supseteq\sigma_1\supseteq \{\be_{2t_2}, \dots, \be_{st_s}\} \qwith \{t_2, \dots, t_s\}=\begin{cases}
\{1,j\} &\text{with $s<j\le q$} \quad \text{or} \\
\{j\} &\text{with $j\in [q]$.} 
\end{cases} \,\, .
\end{equation}
If (i) holds, then $\be_{i1}\notin \sigma_1$, and hence $t_i\ne 1$. In particular, we must have $j>1$. If $j\le s$, then \eqref{e:tau1} implies $\be_{jj}\in \sigma_1$, a contradiction. If $j>s$, then \eqref{e:tau1} implies $\be_{aj}\in\sigma_1$ for some $a$ with $2\le a\le s$, a contradiction. 
 Thus (ii) must hold and in particular $\sigma_1\in G_1$. The arrow $\sigma_1\to \sigma_2$ in the gradient path implies $\sigma_2\to \sigma_1\in A$. By \cref{r:arrowsinG1} we have then $\sigma_2=\sigma_1\cup\{\be_{11}\}= \{ \be_{11},\be_{21},\be_{31},\dots, \be_{s1}\}\cup \gamma'.$

We have $\sigma_3=\sigma_2\ssm \{\be_{1\ell}\}$ for some $\be_{1\ell}\in \sigma_2$ and we may assume $\ell\ne 1$ to avoid repetition in the gradient path. If $\sigma_3\ne \sigma$, then we must have  $\sigma_4\to \sigma_3\in A$. This yields a contradiction, as \cref{r:arrowsinG1} implies $\be_{11}\notin \sigma_3$. Hence $\sigma_3=\sigma$ is a critical cell, and in particular  it must be a critical cell of type (a). It follows that  $2\le \ell \le s$. 
 Note that since $\gamma = \{\be_{ik}\}$, $\sigma$ has the form specified in case (3).

This finishes the proof of the direct implications in (1)--(3). To finalize the proof of the reverse implication in (3), observe that when $|\gamma|=1$ and $\tau$ is of type (b) we have a gradient path
$$\begin{array}{cccccccc}
   \tau &&&& \sigma_2=(\tau\smallsetminus\gamma) \cup \{\be_{11}\} && \\
  &\searrow  && \nearrow && \searrow &\\ 
    &&\sigma_1=\tau\smallsetminus\gamma &&&& \sigma_3=((\tau\smallsetminus\gamma) \cup \{\be_{11}\})\smallsetminus\{\be_{1 \ell}\}=\sigma .
    \end{array}$$
where the fact that $\sigma_2\to \sigma_1\in A$ follows from \cref{r:arrowsinG1}. Thus, when $2\le \ell\le s$, we have $\sigma_A\le \tau_A$. 
\end{proof}

\section{\bf Geometric realizations when $q=4$ and $r=2$}
\label{s:geometric}

It was shown in \cref{t:L1q-prune} and illustrated in \cref{running_ex2,two relations} that if an ideal $I$ satisfies relations of the form 
$$\D = \big \{(1, \{2, \ldots, s\}) \big \} \cup \big \{(b, \{2, \ldots, s\})\colon b\in J \big \} \qforsome J\subseteq \{s+1,\dots, q\},$$
then a single complex supports a resolution of all such $I$, and this complex supports a minimal free resolution of $\ED$ regardless of the size of $J$.  
This complex is illustrated in the left hand diagram of \cref{f:two-triangles-edge}~(when $q=4$) and the right hand diagram in \cref{f:two-edges}~(when $q=3$).
 In \cref{t:minimality}, we constructed a specific simplicial complex $\LDsq$ for the case  when $J=\emptyset$ and $\D=\{(1,\{2,\ldots,s\})\}$, and showed that  $\LDsq$ supports a free resolution of $I^2$ for any $I$ satisfying the relation in $\D$, and the resolution is minimal for $\EDsq$. 

The main focus of this section is offering a geometric realization of the Morse complex $\LDsq$ in the case where $q=4$ and $\D$ consists of a single divisibility relation $\D = \{(1,\{2,3\})\}$.

\begin{example}\label{ex:6-1}
Let
$$
I=(m_1,m_2,m_3,m_4)
\qwhere
m_1 \mid \lcm(m_2,m_3)\, .
$$
Recall from \cref{ex:MorseL2} that the cells of the Morse complex correspond to critical faces of the matching described in \cref{t:L2q-prune}, which are $\{12,13,23\}$ and $\{12,13,14,23\}$ together with the faces of $\LL_4^2$ that do not contain any face in $N$. 
The critical faces of size at least one are listed below, and illustrated in \cref{f:L24-cut-down}. The two critical faces that correspond to non-simplicial cells in the Morse complex are underlined.
{\small
\begin{flalign*}
\text{\bf 3-cells: } & 
\underline{\{12,13,14, 23\}}, 
\{13,14,23,34\}, 
\{12,14,23,24\}\, ; \\
\text{\bf 2-cells: } &
\underline{\{12,13,23\}}, 
\{13,14,34\},
\{13,14,23\},
\{13, 14, 11\},
\{12,14,24\},
\{12,14,23\},
\{12,14,11\}, 
\{14,23,34\}\, ;\\ 
&\{14,23,24\},
\{14,34,44\},
\{14,24, 44\},
\{12,24,22\},
\{12,23,24\},
\{13,34,33\}, 
\{13,23,34\}\, ;\\
\text{\bf Edges: } &
\{12,14\},
\{12,23\},
\{12,24\},
\{12,11\},
\{12,22\},
\{13,14\},
\{13,23\},
\{13, 34\},
\{13,11\},
\{13,33\},
\{14,24\}\, ;\\
&\{14,23\},
\{14,34\},
\{14,11\},
\{14,44\},
\{23,24\},
\{23,34\},
\{24,22\},
\{24,44\},
\{34,33\},
\{34,44\}\, ;\\
\text{\bf Vertices: } &
\{11\},
\{12\},
\{13\},
\{14\},
\{22\},
\{23\},
\{24\},
\{33\},
\{34\},
\{44\}\, .
\end{flalign*}
}
The cells corresponding to critical faces found from $Y\smallsetminus N$ are simplices. This is because every subset of such a face is again a critical face in the list and so $(\sigma \smallsetminus \{ij\})_A \le \sigma_A$ for these faces whenever $\be_{ij}\in \sigma$, by \cref{t:cell order}.
However, the cells corresponding to the underlined faces have subcells relative to the order $\le$ that come from gradient paths as described in \cref{t:cell order} case (3), not merely from inclusions. In particular, $\{11,12,14\}_A \le \{12,13,14,23\}_A$ and $\{11,13,14\}_A \le \{12,13,14,23\}_A$.

The face $\{12,13,23\}$ also has subcells arising from gradient paths, in particular $\{11,12\}_A \le \{12,13,23\}_A$ and $\{11,13\}_A \le \{12,13,23\}_A$.
Thus the cell $\{12,13,23\}_A$ can be visualized as a square with vertices $11,12,13,23$ and edges $\{11,12\}_A$, $\{11,13\}_A$, $\{13, 23\}_A$ and $\{23,12\}_A$. Note that $\{12,13\}$ was used in the matching and so will not generate a cell. Similarly, the cell $\{12,13,14,23\}_A$ can be visualized as a pyramid whose faces are the square $\{12,13,23\}_A=\{11,12,13,23\}_A$, and the triangles $\{11,12, 14\}_A$, $\{11,13,14\}_A$, $\{13,14, 23\}_A$ and $\{23,12,14\}_A$. 

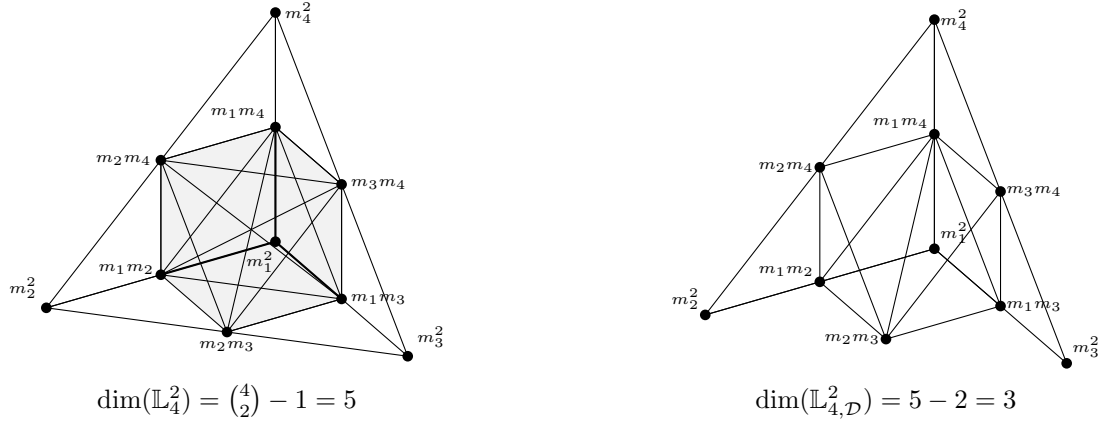
\begin{figure}[!htbp]
        
$$\begin{array}{ccc}
\tdplotsetmaincoords{60}{150}

\begin{tikzpicture}[tdplot_main_coords]

    \coordinate (A) at (0,0,0);
    \coordinate (B) at (1.75,0,0);
    \coordinate (C) at (3.5,0,0);
    \coordinate (D) at (0,1.75,0);
    \coordinate (E) at (0,3.5,0);
     \coordinate (F) at (1.75,1.75,0);
      \coordinate (G) at (0,0,1.75);
      \coordinate (H) at (0,0,3.5);
      \coordinate (I) at (1.75,0,1.75);
      \coordinate (J) at (0,1.75,1.75);
\draw [fill=gray!10] (B.center) -- (I.center) -- (G.center) -- (J.center)-- (D.center) -- (F.center);
      
  \draw (F) -- (B) ;
  \draw (F) -- (D);
   \draw [thick] (A) -- (B);
      \draw [thick] (A) -- (D);   
    \draw [thick] (A) -- (G);
       \draw (G) -- (H);
          \draw (G) -- (J);
          \draw (G) -- (I);
          \draw (B) -- (I);
          \draw (G) -- (B);
          \draw (B) -- (C);
          \draw (I) -- (C);
\draw (H) -- (I);
\draw (H) -- (J); 
\draw (J) -- (D);  
\draw (E) -- (J);
\draw (F) -- (J);
\draw (G) -- (F);
\draw (I) -- (F);
\draw (G) -- (D);

\draw (I) -- (J);
\draw (C) -- (F);
\draw (F) -- (E);
\draw (B) -- (D);
\draw (I) -- (D);
\draw (B) -- (J);

    \foreach \point in {A, B, C, D, E, F, G, H, I, J} \fill (\point) circle (2pt);

    \node[below left, xshift=3pt] at (A) {\tiny{$m_1^2$}};
    \node[above, left, xshift=1pt, yshift = 2pt] at (B) {\tiny{$m_1m_2$}};
    \node[above left] at (C) {\tiny{$m_2^2$}};
    \node[right] [label distance=-6pt]  at (D) {\tiny{$m_1m_3\ \ $}};
    \node[above right] at (E) {\tiny{$m_3^2$}};
    \node[ below] at (F) {\tiny{$m_2m_3$}};
     \node[above left] at (G) {\tiny{$m_1m_4$}};
    \node[right] at (H) {\tiny{$m_4^2$}};
    \node[left] at (I) {\tiny{$m_2m_4$}};
     \node[right] at (J) {\tiny{$m_3m_4$}};


    \draw[-] (0,0,0) -- (3.5,0,0) node[anchor=north east]{};
    \draw[-] (0,0,0) -- (0,3.5,0) node[anchor=north west]{};
    \draw[-] (0,0,0) -- (0,0,3.5) node[anchor=south]{};
\end{tikzpicture} 
&
\qquad \qquad \qquad
&
{\tiny \tdplotsetmaincoords{60}{150}
\begin{tikzpicture}[tdplot_main_coords]

    \coordinate (A) at (0,0,0);
    \coordinate (B) at (1.75,0,0);
    \coordinate (C) at (3.5,0,0);
    \coordinate (D) at (0,1.75,0);
    \coordinate (E) at (0,3.5,0);
     \coordinate (F) at (1.75,1.75,0);
      \coordinate (G) at (0,0,1.75);
      \coordinate (H) at (0,0,3.5);
      \coordinate (I) at (1.75,0,1.75);
      \coordinate (J) at (0,1.75,1.75);
  \draw (F) -- (B) ;
  \draw (F) -- (D);
   \draw (A) -- (B);
      \draw (A) -- (D);   
    \draw (A) -- (G);
       \draw (G) -- (H);
          \draw (G) -- (J);
          \draw (G) -- (I);
          \draw (B) -- (I);
          \draw (G) -- (B);
          \draw (B) -- (C);
          \draw (I) -- (C);
\draw (H) -- (I);
\draw (H) -- (J); 
\draw (J) -- (D);  
\draw (E) -- (J);
\draw (F) -- (J);
\draw (G) -- (F);
\draw (I) -- (F);
\draw (G) -- (D);

    \foreach \point in {A, B, C, D, E, F, G, H, I, J} \fill (\point) circle (2pt);

    \node[above right] at (A) {\tiny{$m_1^2$}};
    \node[above left] at (B) {\tiny{$m_1m_2$}};
    \node[above left] at (C) {\tiny{$m_2^2$}};
    \node[right] [label distance=-6pt]  at (D) {\tiny{$m_1m_3\ \ $}};
    \node[above right] at (E) {\tiny{$m_3^2$}};
    \node[ left] at (F) {\tiny{$m_2m_3$}};
     \node[above left] at (G) {\tiny{$m_1m_4$}};
    \node[right] at (H) {\tiny{$m_4^2$}};
    \node[left] at (I) {\tiny{$m_2m_4$}};
     \node[right] at (J) {\tiny{$m_3m_4$}};


    \draw[-] (0,0,0) -- (3.5,0,0) node[anchor=north east]{};
    \draw[-] (0,0,0) -- (0,3.5,0) node[anchor=north west]{};
    \draw[-] (0,0,0) -- (0,0,3.5) node[anchor=south]{};

\end{tikzpicture} }\\

\dim(\LL_4^2) = \binom{4}{2}-1=5 & \qquad \qquad \qquad & \dim(\LL_{4,\D}^2) = 5-2 =3

\end{array}$$
    \caption{Left: \ \ $\LL_4^2$ $\qquad$ Right: \ \ $\LL_{4,\D}^2$ when $\D=\{(1,\{2,3\})\}$ }
    \label{f:L24-cut-down}
\end{figure}

These results are illustrated in Figure \ref{f:L24-cut-down}. On the left is the simplicial complex $\LL^2_4$. On the right is a geometric representation of the cells of the Morse complex $\LL^2_{4,\D}$. This complex will support a resolution of $I^2$, when a vertex $\be_{ij}$ is labeled by $m_im_j$.
The square and pyramid described above appear in the depiction, along with the two tetrahedra appearing in $Y\smallsetminus N$. The pyramid and tetrahedra are three dimensional and all other cells are of lower dimension. 

For comparison purposes, we list the Betti numbers for the ${\E_4}^2$ and $\E_{4,\D}^{\,\,2}$ provided by resolutions supported on $\LL^2_4$ and $\LL_{4,\D}^2$ and note that the differences correspond to faces matched in \cref{ex:MorseL2}.

\begin{table}[h!] 
\begin{tabular}{|c|c|c|c|c|c|c|}
\hline
     &  $\beta_0 \leq$ & $\beta_1 \leq$ & $\beta_2 \leq$ & $\beta_3 \leq$ & $\beta_4 \leq$ & $\beta_5\leq$ \\ \hline
 $\LL^2_4$    & 10 & 27 & 32 & 19 & 6 & 1 \\ \hline
 $\LL^2_{4,\D}$ & 10 & 21 & 15 & 3 & 0 & 0 \\ \hline
\end{tabular}
\caption{The numbers of $i$-cells of $\LL^2_4$ and $\LL_{4,\D}^2$, bounding  $\beta_i(I^2)$ when
$I=(m_1,m_2,m_3,m_4)$
and
$m_1 \mid \lcm(m_2,m_3)$. In the cases  $I=\E_4$ and $I=\E_{4,\D}$, equality holds in the respective lines. } \label{t:table}
\end{table} 
\end{example}

 By \cref{t:minimality} and \cite[Prop.~7.10]{Lr}, when $|\D| \leq 1$, $\LL_{4,\D}^2$ supports a minimal resolution of $\E_{4,\D}^{\,\,2}$. For other square-free monomial ideals whose generators satisfy the relations in $\D$, $\LL_{4,\D}^2$ supports a resolution that is much closer to minimal than that of the Taylor complex or of $\LL^2_4$. Note that in general, the resolution need not be minimal, even if $I$ satisfies only the relations in $\D$ as there may be additional relations on $I^2$, as seen in the first example below. In addition, since $q=4$ and $s=3$ for each of these examples, by \cref{t:betti-2-bound}, we have an upper bound of $\binom{4}{2} - (4-3+1) =3$ for the projective dimension, which is achieved for all but the final example.

\begin{example} Recall from \cref{running_ex2} that 
$$
I_1=(ab,bcd,aef,cg), \quad 
(\beta_0({I_1}^2),\beta_1({I_1}^2),\beta_2({I_1}^2), \beta_3({I_1}^2))= (10, 17, 9, 1)\, ,
$$
and the only minimal divisibility relation that $I_1$ satisfies is the one in  $\D= \{(1, (2,3)\}$. The complex $\LL^2_{4,\D}$ supports a free resolution of ${I_1}^2$, and upper bounds for the Betti numbers of $I_1^2$ can be read from the corresponding row of  \cref{t:table}. However, additional relations exist for ${I_1}^2$. 
 For example, it is straightforward to check that the following relations are satisfied by the generators of ${I_1}^2$:
\begin{align*}
m_1m_4  \mid \lcm(m_1m_2, m_2m_4) \quad & m_1m_4  \mid \lcm(m_2m_3, m_2m_4)\\
m_3m_4  \mid \lcm(m_2m_3, m_1m_4) \quad & m_1m_4  \mid \lcm(m_1m_3, m_3m_4)\,.
\end{align*}
Using the additional relations, the cell complex in \cref{f:L24-cut-down} can be further reduced. 
It can be verified using Macaulay2 \cite{M2} that the left hand figure shown in \cref{temp2} supports a minimal resolution of ${I_1}^2$. Note that in this depiction, one of the 2-cells is non-planar. In particular, the vertices labeled $m_2m_3, m_3m_4, m_1m_4, m_1m_3$ form a cell, which is a boundary cell of the unique 3-dimensional cell.
\end{example}

We conclude with a final example, showing that as above, for a general ideal $I$ satisfying the relations $\D'=\{(1,\{2,3\}),(4,\{2,3\})\}$, the resolution supported on $\LDsq$ will not be minimal, but $\LDsq$ still provides effective bounds for the Betti numbers.

\begin{example}[{\bf Two divisibility relations}] \label{two relations redux}  
Unlike the situation for the first power given in \cref{t:L1q-prune} and illustrated in \cref{running_ex2,two relations}, when $I$ satisfies additional relations, the resolution supported on $\LDsq$ will in general no longer be minimal. Even for $q=4$, 
$$
\D=\{(1,\{2,3\})\} 
\qand 
\D'=\{(1,\{2,3\}),(4,\{2,3\})\}\,,
$$
the resolution supported on $\LL_{4,\D}^2$ will not be minimal for the ideal $\E^{\,\,2}_{\D'}$. Essentially, for the first power, once a Morse matching has been completed on the Taylor complex for $I$ using the relation $(1, \{2, \ldots, s\})$, enough faces have been removed that no additional matching based on $(b, \{2, \ldots, s\})$ remains. However, when starting with $I^2$, the additional divisibility relation in $\D'$ can be used to further reduce $\LL_{4,\D}^2$. Computation in Macaulay2 \cite{M2} shows that 
$$
(\beta_0(\E_{4,\D'}^{\,\, 2}),\beta_1(\E_{4,\D'}^{\,\, 2}),\beta_2(\E_{4,\D'}^{\,\, 2}), \beta_3(\E_{4,\D'}^{\,\, 2})) 
= 
(10,21,14,2),
$$ 
indicating that one 2-dimensional cell has been matched with one 3-dimensional cell. Noting that for $\E_{\D'}$ we have $m_1m_4 \mid \lcm(m_1^2, m_1m_2, m_1m_3,m_2m_3)$, one can see, and verify via the differentials 
computed using Macaulay2, that it is precisely the square and pyramid described above that are matched to reduce $\LDsq$ to a structure, which is simplicial, supporting a minimal resolution of $\E^{\,\,2}_{\D'}$. 

For example the ideal
$$
I_2=(ab, bcd, aef, ce)
$$ 
 in \cref{two relations} satisfies the divisibility relations in $\D'$, but  $m_1m_4 \mid m_2m_3$  and thus even the minimal number of generators, $\beta_0({I_2}^2)$, is strictly less than the bound provided by $\LDsq$. Using Macaulay2~\cite{M2} we see that
$$
(\beta_0({I_2}^2),\beta_1({I_2}^2),\beta_2({I_2}^2), \beta_3({I_2}^2))= (9,14,6,0),
$$
and we can  verify via the differentials that the figure on the right hand side of \cref{temp2} supports a resolution of ${I_2}^2$. 
\end{example}

Below are the figures illustrating simplicial complexes that support a minimal free resolution for ${I_1}^2$ and ${I_2}^2$ from \cref{running_ex2,two relations}.  The figure on the right shows 4 triangles and 2 quadrilaterals with no cells of dimension greater than two.  The figure on the left has 4 triangles and the 5 quadrilaterals, with one 3-dimensional cell outlined using thicker edges.
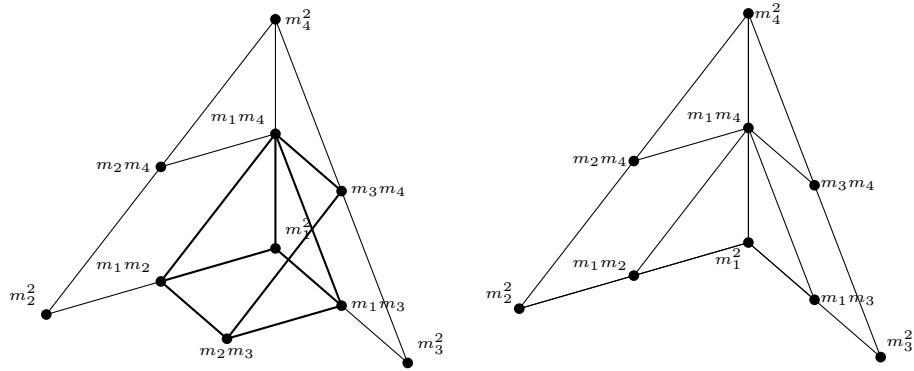
\begin{figure}[!htbp]
    
$$\begin{array}{cc}
\tdplotsetmaincoords{60}{150}

\begin{tikzpicture}[tdplot_main_coords]

    \coordinate (A) at (0,0,0);
    \coordinate (B) at (1.75,0,0);
    \coordinate (C) at (3.5,0,0);
    \coordinate (D) at (0,1.75,0);
    \coordinate (E) at (0,3.5,0);
     \coordinate (F) at (1.75,1.75,0);
      \coordinate (G) at (0,0,1.75);
      \coordinate (H) at (0,0,3.5);
      \coordinate (I) at (1.75,0,1.75);
      \coordinate (J) at (0,1.75,1.75);

  \draw [thick](F) -- (B) ;
  \draw [thick](F) -- (D);
   \draw [thick] (A) -- (B);
      \draw [thick](A) -- (D);   
    \draw [thick](A) -- (G);
       \draw (G) -- (H);
          \draw [thick](G) -- (J);
          \draw (G) -- (I);
          \draw [thick](G) -- (B);
          \draw (B) -- (C);
          \draw (I) -- (C);
\draw (H) -- (I);
\draw (H) -- (J); 
\draw (E) -- (D);  
\draw (E) -- (J);
\draw [thick](F) -- (J);
\draw [thick](G) -- (D);

    \foreach \point in {A, B, C, D, E, F, G, H, I, J} \fill (\point) circle (2pt);

    \node[above right] at (A) {\tiny{$m_1^2$}};
    \node[above left] at (B) {\tiny{$m_1m_2$}};
    \node[above left] at (C) {\tiny{$m_2^2$}};
    \node[right] [label distance=-6pt]  at (D) {\tiny{$m_1m_3\ \ $}};
    \node[above right] at (E) {\tiny{$m_3^2$}};
    \node[ below] at (F) {\tiny{$m_2m_3$}};
     \node[above left] at (G) {\tiny{$m_1m_4$}};
    \node[right] at (H) {\tiny{$m_4^2$}};
    \node[left] at (I) {\tiny{$m_2m_4$}};
     \node[right] at (J) {\tiny{$m_3m_4$}};

\end{tikzpicture} 
&{\tiny \tdplotsetmaincoords{60}{150}
\begin{tikzpicture}[tdplot_main_coords]

    \coordinate (A) at (0,0,0);
    \coordinate (B) at (1.75,0,0);
    \coordinate (C) at (3.5,0,0);
    \coordinate (D) at (0,1.75,0);
    \coordinate (E) at (0,3.5,0);
      \coordinate (G) at (0,0,1.75);
      \coordinate (H) at (0,0,3.5);
      \coordinate (I) at (1.75,0,1.75);
      \coordinate (J) at (0,1.75,1.75);

   \draw (A) -- (B);
      \draw (A) -- (D);   
    \draw (A) -- (G);
       \draw (G) -- (H);
          \draw (G) -- (J);
          \draw (G) -- (I);
          \draw (G) -- (B);
          \draw (B) -- (C);
          \draw (I) -- (C);
\draw (H) -- (I);
\draw (H) -- (J); 
\draw (E) -- (J);
\draw (G) -- (D);

    \foreach \point in {A, B, C, D, E, G, H, I, J} \fill (\point) circle (2pt);

    \node[below left] at (A) {\tiny{$m_1^2$}};
    \node[above left] at (B) {\tiny{$m_1m_2$}};
    \node[above left] at (C) {\tiny{$m_2^2$}};
    \node[right] [label distance=-6pt]  at (D) {\tiny{$m_1m_3\ \ $}};
    \node[above right] at (E) {\tiny{$m_3^2$}};
     \node[above left] at (G) {\tiny{$m_1m_4$}};
    \node[right] at (H) {\tiny{$m_4^2$}};
    \node[left] at (I) {\tiny{$m_2m_4$}};
     \node[right] at (J) {\tiny{$m_3m_4$}};

 \coordinate (T1) at (.5,0,.5);
    \coordinate (T3) at (0,.5,.5);
    \coordinate (T2) at (0,0,2.5);
    \coordinate (T4) at (0,0,2.2);
    \coordinate (P1) at (1.4,0,.8);
    \coordinate (P2) at (0,1,.7);

    \draw[-] (0,0,0) -- (3.5,0,0) node[anchor=north east]{};
    \draw[-] (0,0,0) -- (0,3.5,0) node[anchor=north west]{};
    \draw[-] (0,0,0) -- (0,0,3.5) node[anchor=south]{};

\end{tikzpicture} }\\

\end{array}$$
    \caption{Cell complexes supporting a minimal resolution of ${I_1}^2$ (left) and ${I_2}^2$ (right)}
    \label{temp2}
\end{figure}

\bibliographystyle{plain}
\bibliography{refs.bib}

\end{document}